\documentclass[11pt]{amsart}
\usepackage{amssymb}
\usepackage{amsmath}
\usepackage{amsfonts}
\usepackage{mathrsfs}
\usepackage{psfrag}
\usepackage{verbatim}
\usepackage{color}
\usepackage{bm}
\usepackage[usenames,dvipsnames]{xcolor}
\usepackage{graphicx} 
\usepackage{bm}
\usepackage{bbm}
\usepackage{dsfont}
\usepackage{enumerate}
\usepackage{marginnote}
\usepackage{tcolorbox}
\usepackage[colorlinks=true]{hyperref}
\hypersetup{citecolor=MidnightBlue}
\newcommand{\ind}{\mathds{1}}

\DeclareGraphicsExtensions{}
\theoremstyle{plain}
\newtheorem{theorem}{Theorem}[section]

\newtheorem{lemma}[theorem]{Lemma}
\newtheorem{proposition}[theorem]{Proposition}

\theoremstyle{definition}
\newtheorem{assumption}[theorem]{Assumption}
\newtheorem{definition}[theorem]{Definition}
\newtheorem{remark}[theorem]{Remark}
\theoremstyle{remark}

\newcommand{\eps}{\varepsilon}

\newcommand{\BR}{\mathbb{R}}
\newcommand{\BP}{\mathbb{P}}
\newcommand{\BE}{\mathbb{E}}
\newcommand{\BZ}{\mathbb{Z}}
\newcommand{\BN}{\mathbb{N}}

\newcommand{\DD}{\mathsf{D}} 

\newcommand{\sfR}{\mathsf{R}}
\newcommand{\sfV}{\mathsf{V}}
\newcommand{\sfr}{\mathsf{r}}
\newcommand{\sfv}{\mathsf{v}}
\newcommand{\sfT}{\mathsf{T}} 
\newcommand{\itd}{\mathtt{Id}} 
\newcommand{\sfN}{\mathsf{N}} 
\newcommand{\sfK}{\mathsf{K}}
\newcommand{\sfw}{\mathsf{w}}

\newcommand{\II}{\mathtt{I}}
\newcommand{\JJ}{\mathtt{J}}

\newcommand{\LL}{\mathtt{L}}
\newcommand{\MM}{\mathtt{M}}
\newcommand{\NN}{\mathtt{N}}
\newcommand{\lhp}{\varrho}

\newcommand{\ren}{\mathtt{q}}
\newcommand{\Ren}{\mathtt{Q}}

\newcommand{\filt}{\mathscr{F}}

\begin{document}
\title[Randomly perturbed system with state-dependent impulse effects]{Fluctuation analysis for a randomly perturbed dynamical system with state-dependent impulse effects}
\author{Ashif Khan and Chetan D. Pahlajani}
\address{Mathematics\\ Indian Institute of Technology Gandhinagar}
\address{Department of Mathematics\\ Indian Institute of Technology Gandhinagar}
\email{khanashif@iitgn.ac.in, cdpahlajani@iitgn.ac.in}
\date{\today.}
\thanks{The first author would like to thank Government of India for financial support through the Prime Minister's Research Fellowship (PMRF) scheme. The second author's research was supported by ANRF project number MTR/2023/000545.}
\maketitle

{\small \centerline{Department of Mathematics, Indian Institute of Technology Gandhinagar, Palaj, Gandhinagar 382055, India}}

\maketitle

\begin{abstract}
The principal aim of the present work is to explore limit theorems for small random perturbations of a planar impulsive dynamical system, where impulses occur at hitting times of a suitable switching surface, and are thus state-dependent. Working with a simplified example in polar coordinates, we obtain---for any fixed time horizon---a small noise expansion for the radial component, together with rigorous error estimates in the Skorohod space of right-continuous functions with left limits. 
\end{abstract}


\section{Introduction}\label{S:Introduction}

\color{black} In many scientific and engineering applications, one encounters \textit{systems with impulse effects} ({\sc sie}), where the continuous evolution of the state according to an ordinary differential equation ({\sc ode}) is punctuated by abrupt changes of state, governed by \textit{resetting maps}, at discrete impulse times. An elementary example to keep in mind is a simple mass-spring oscillator which makes intermittent contact with a rigid barrier: the evolution between impacts is governed by a linear second-order {\sc ode} encoding Newton's laws, with the velocity being instantaneously reversed, perhaps with some energy loss, upon collision with the barrier. Such scenarios are, not surprisingly, ubiquitous in mechanical systems; see \cite{Filippov}, \cite{GAP}, \cite{BristolBook}, \cite{Brogliato-book} for a wealth of information about theoretical formulation, novel features like discontinuity-induced bifurcations, and applications (e.g., robotics).

\color{black} Since most realistic systems are subject to uncertainty, either due to external disturbances or imperfect modeling, stochastic process models---frequently, stochastic differential equations ({\sc sde}) driven by Brownian motion---are used to include the effects of random perturbations. Since the latter are typically small, much attention has been devoted to asymptotic techniques like small noise expansions, averaging, and large deviations \cite{FW_RPDS}, \cite{SHS-RPM}, \cite{DZ98}, which allow one to extract simpler approximate models or estimate probabilities of extremely unlikely behavior. While these techniques have been extensively studied for problems with smooth underlying dynamics, the corresponding calculations have not been adequately explored for {\sc sie}.  

\color{black} In a recent work \cite{KhanPahlajani}, small noise expansions were explored for small random perturbations of a reasonably wide class of {\sc ode} subject to \textit{periodic} impulse effects: the framework allowed nonlinearity in both the governing {\sc ode} between impulses and the resetting map at impulses.  
The zeroth- and first-order terms of the expansion were computed (along with rigorous error estimates), and shown to correspond to the limiting mean behavior and typical fluctuations about the mean in the limit of vanishing noise. 
The novel finding in \cite{KhanPahlajani} was the description of the fluctuation process, which involved linearizing the {\sc sde} in the vicinity of the deterministic trajectory in between impulses, similar to \cite{Blagoveshchenskii}, and also a linearization of the resetting map at impulses. 
A key feature that facilitated the analysis was the assumption that the impulse times in the perturbed stochastic system were the \textit{same} as those in the unperturbed deterministic system. 

\color{black} The principal aim of the present work is explore similar questions for small random perturbations of {\sc sie} where impulses arrive, not in a time-dependent (e.g., periodic), but rather \textit{state-dependent}, manner. 
More precisely, impulses are assumed to occur every time the solution of the {\sc ode} (in the deterministic case) or {\sc sde} (in the stochastic case) hits a certain switching surface.
To understand a prototypical problem, suppose we have a smooth vector field $b:\BR^d \to \BR^d$, a smooth codimension one surface $\mathscr{S} \subset \BR^d$, and a continuous map $\Delta:\mathscr{S} \to \BR^d$. 
For a function $x:[0,\infty) \to \BR^d$ which possesses right-hand limits for all $t \ge 0$ and left-hand limits for $t>0$, we define the functions $x^+(t)$ and $x^-(t)$ by $x^+(t)\triangleq \lim_{s \searrow t}x(s)$ for all $t \ge 0$, $x^-(t)\triangleq \lim_{s \nearrow t}x(s)$ for all $t>0$, $x^-(0) \triangleq x(0)$. 
A generic {\sc sie} with state-dependent impulse effects takes the form
\begin{equation}\label{E:sie-generic-det}
\begin{aligned}
\varSigma: \begin{cases}\dot{x}(t) &= b(x(t)) \qquad \text{for $x^-(t) \notin \mathscr{S}$}\\x^+(t) &= \Delta(x^-(t)) \quad \text{for $x^-(t) \in \mathscr{S}$.}\end{cases}
\end{aligned}
\end{equation}
A solution to the {\sc sie} \eqref{E:sie-generic-det} is now a function $x:[0,\infty) \to \BR^d$ which is right-continuous with left limits satisfying the {\sc ode} $\dot{x}=b(x)$ while away from $\mathscr{S}$ and which is reset using the map $x \mapsto \Delta(x)$ at (or more accurately, just before) impact with $\mathscr{S}$; see \cite{GAP}.

\color{black} If the system above is subject to small random perturbations of size $\eps$ with $0<\eps \ll 1$, then the evolution should be governed by a stochastic process $X^\eps_t$ whose paths are right-continuous with left limits, and solves
\begin{equation}\label{E:sie-generic-stoch}
\begin{aligned}
\varSigma^\eps: \begin{cases}dX^\eps_t &= b(X^\eps_t)\thinspace dt + \eps   \sigma^\eps(X^\eps_t) \thinspace dW_t \qquad \text{for $X^{\eps,-}_t \notin \mathscr{S}$}\\X^{\eps,+}_t &= \Delta(X^{\eps,-}_t) \qquad \qquad \qquad \medspace \qquad \text{for $X^{\eps,-}_t \in \mathscr{S}$,}\end{cases} \\
 \text{where} \quad \sigma^\eps(x) = \sigma_0(x) + \eps^{p-1} \sigma_1(x) \quad \text{with $p > 1$.}&
\end{aligned}
\end{equation}
Once again, $X^{\eps,+}_t \triangleq \lim_{s \searrow t}X^\eps_s$ for $t \ge 0$ and $X^{\eps,-}_t \triangleq \lim_{s \nearrow t}X^\eps_s$ for $t>0$, $X^{\eps,-}_0 \triangleq X^\eps_0$ are the right- and left-continuous modifications of $X^\eps_t$, defined path-by-path in the manner described above. 
Similar to the case above, the process $X^\eps_t$ evolves according to the {\sc sde} $dX^\eps_t = b(X^\eps_t)\thinspace dt + \eps   \sigma^\eps(X^\eps_t) \thinspace dW_t$ while away from $\mathscr{S}$, and is reset according to $X^{\eps,+}_t = \Delta(X^{\eps,-}_t)$ just before hitting $\mathscr{S}$.
The function $\sigma^\eps:\BR^d \to \BR^{d\times r}$ is assumed to be sufficiently regular and $W_t$ is an $r$-dimensional Brownian motion on some probability space $(\Omega,\filt,\BP)$.
We have taken the diffusion matrix $\sigma^\eps(x)$ to be of the form $\sigma_0(x) + \eps^{p-1} \sigma_1(x)$ with $p > 1$ since this will enable us, by carefully tuning $\sigma_0$ and $\sigma_1$, to introduce noise at two different scales, viz., $\eps$ and $\eps^p$, in different components of $X^\eps_t$. 

\color{black} To understand the type of questions we would like to ask, let us consider for a moment the equations \eqref{E:sie-generic-det} and \eqref{E:sie-generic-stoch} in the \textit{absence} of impulse effects, i.e., when $\mathscr{S}=\emptyset$, with $\sigma_1(x) \equiv 0$. 
In this setting, classical results \cite{FW_RPDS} imply that for any fixed time horizon $\sfT>0$, the error $\sup_{0 \le t \le \sfT}|X^\eps_t-x(t)|$ converges to zero in probability as $\eps \searrow 0$; this can be viewed as a functional \textit{law of large numbers} ({\sc lln}). 
If the drift and diffusion coefficients are sufficiently smooth, one can next obtain the expansion
$X^\eps_t = x(t) + \eps X^{(1)}_t + \mathscr{O}(\eps^{2})$, along with rigorous error estimates, with the governing equation for $X^{(1)}_t$ being explicitly computable by linearization.
The resulting assertion that $\eps^{-1}(X^\eps_t - x(t))$ converges in distribution to $X^{(1)}_t$ can be interpreted as a functional \textit{central limit theorem} ({\sc clt}). 
In both these cases, the sample paths of $X^\eps_t$ are compared to either $x(t)$ or $x(t)+\eps X^{(1)}_t$ in the space $C([0,\sfT];\BR^d)$ consisting of all continuous functions mapping $[0,\sfT]$ to $\BR^d$, equipped with the the sup norm.
Further, the size of the error predictably involves one power of $\eps$ higher than the largest power retained in the approximation.

\color{black} In a nutshell, we would like to execute a similar program for \eqref{E:sie-generic-det} and \eqref{E:sie-generic-stoch} in the \textit{presence} of impulse effects, i.e., when $\mathscr{S} \neq \emptyset$. 
It is not hard to see that, absent additional simplifying assumptions, such problems pose significant challenges, the most prominent of which is the difficulty of getting good \textit{quantitative} information about the impulse (hitting) times.  
To start, there do not appear to be any analytical techniques to compute the hitting times for a general noiseless system of the form \eqref{E:sie-generic-det}. 
In the presence of noise, one can use the Markovian property of the process $X^\eps_t$ solving \eqref{E:sie-generic-stoch} ``starting afresh" after every resetting to relate the times between impulses to first passage times for small noise diffusion processes.
While the resulting random hitting times may be expected to be close to their deterministic counterparts with high probability (on account of smallness of the noise), one still needs sharp enough estimates for the probabilities of stochastic hitting times deviating from their deterministic counterparts by various amounts.
Finally, we note that the function $x(t)$ and the sample paths of the process $X^\eps_t$ do not belong to $C([0,\sfT];\BR^d)$, but rather to the Skorohod space $D([0,\sfT];\BR^d)$ consisting of all functions mapping $[0,\sfT]$ to $\BR^d$ which are right-continuous with left limits, equipped with the Skorohod metric \cite{ConvProbMeas}, \cite{EK86}.
Questions of convergence should thus be studied in $D([0,\sfT];\BR^d)$, thereby adding an added layer of complexity in comparison to the classical case.\footnote{We note that while the corresponding quantities in \cite{KhanPahlajani} also had paths in $D([0,\sfT];\BR^d)$, we were able to get away with using the sup norm on account of the fact that in \cite{KhanPahlajani}, the hitting times in the deterministic and stochastic contexts were the same.}

\color{black} As a first step towards understanding limit theorems for general systems of the form \eqref{E:sie-generic-det} and \eqref{E:sie-generic-stoch}, and motivated by an analytic example studied in \cite[Section 5.2]{CBC-PoincareBendixson}, we study here a simple problem in polar coordinates. 
This system, described in Section \ref{S:ProblemStatement}, involves adding small Brownian perturbations (in both radial and angular components) to an {\sc sie} governed by an {\sc ode} in a wedge, together with resetting conditions on exiting the wedge. 
The structural assumptions of constant angular velocity and additive noise easily yield expressions for the deterministic impulse times, and allow expressing their stochastic counterparts in terms of first passage times of Brownian motion with drift. 
Assuming that the noise intensity in the angular coordinate vanishes faster than that in the radial coordinate, we are able to obtain a small noise expansion for the radial component, together with rigorous error estimates in the space $D([0,\sfT];\BR)$ with respect to the Skorohod metric.
Using simply the zeroth-order term, we are able to get error estimates comparable to those in the smooth impulse-free case.
However, the error estimates for the first-order expansion are not as sharp as those in the smooth case. 
One reason is the fact that one is comparing quantities with different---but close (with high probability)---discontinuity times in the Skorohod metric. 
An added factor is perhaps a limitation of our technique which, strictly speaking, estimates not exactly the Skorohod distance, but rather an \textit{upper bound} to it, possibly leading to a loss of sharpness.

\subsection*{Some comments regarding notation}
When deriving various estimates, we will denote generic constants by $C$, allowing the exact value of $C$ to  change from line to line. Further, with the \textit{sole exception} of $\eps$ and any other small parameters $\delta=\delta(\eps)$, $\zeta=\zeta(\eps)$ we may encounter, we will absorb the dependence on \textit{all} problem parameters into $C$. Thus, a typical $C$ will depend on all problem parameters (including time horizon $\sfT$), but \textit{not} on $\eps$.


\section{Problem Formulation and Statement of Main Results}\label{S:ProblemStatement}
\subsection{Setting}\label{SS:Setting}
Let $(r,\theta)$ denote polar coordinates in the plane $\BR^2$, and consider the {\sc ode}
\begin{equation}\label{E:ode-polar}
\frac{dr}{dt} = b(r), \qquad 
\frac{d\theta}{dt} =1, 
\end{equation}
where $b:\BR \to \BR$ is a smooth vector field. Next, we fix $\alpha \in (0,2 \pi)$ and let $\mathscr{R}$, $\mathscr{S}$ be the curves in $\BR^2$ given by $\mathscr{R} \triangleq \{(r,\theta):\theta=0\}$ and $\mathscr{S} \triangleq \{(r,\theta):\theta=\alpha\}$, respectively. Let $\Delta:\mathscr{S} \to \mathscr{R}$ be a smooth, possibly nonlinear, function. Our problem of interest can be described as a \textit{system with impulse effects} ({\sc sie}) \cite{GAP}:
\begin{equation}\label{E:sie-det}
\begin{aligned}
\frac{d}{dt}\left[\begin{matrix}r(t)\\\theta(t)\end{matrix}\right] &= \left[\begin{matrix}b(r(t))\\1\end{matrix}\right] \qquad \text{for $(r^-(t),\theta^-(t)) \notin \mathscr{S}$}\\
(r^+(t),\theta^+(t)) &= \Delta(r^-(t),\theta^-(t)) \qquad \text{for $(r^-(t),\theta^-(t)) \in \mathscr{S}$,}\\
(r(0),\theta(0)) &= (r_0,\theta_0) \in \BR^2\setminus\mathscr{S}.
\end{aligned}
\end{equation} 
If the initial conditions satisfy $r_0>0$, $0 \le \theta < \alpha$, then the evolution of $(r(t),\theta(t))$ is restricted to the wedge $\DD \triangleq \{(r,\theta):r \ge 0, 0 \le \theta \le \alpha\}$.

To summarize the foregoing in words, our system of interest evolves according to the {\sc ode} \eqref{E:ode-polar} as long as it is away from $\mathscr{S}$. Upon (or more accurately, just before) hitting the \textit{impact surface} $\mathscr{S}$, the map $\Delta$ resets the state to a point on the \textit{resetting surface} $\mathscr{R}$. The evolution now resumes according to the {\sc ode} \eqref{E:ode-polar} and the cycle continues. Our problem is thus quite similar to one in \cite{CBC-PoincareBendixson}.

We would now like to study small random perturbations of the system \eqref{E:sie-det} given by
\begin{equation}\label{E:sie-stoch}
\begin{aligned}
\left[\begin{matrix}dR^\eps_t \\d\Theta^\eps_t\end{matrix}\right] &= \left[\begin{matrix} b(R^\eps_t)\\1\end{matrix}\right]\thinspace dt + \zeta\left[\begin{matrix}0\\  f(R^\eps_t,\Theta^\eps_t)\end{matrix}\right] \thinspace dt + \left[\begin{matrix}
	\eps & 0\\
      0  & \eps^{p}
\end{matrix}\right] \left[\begin{matrix}dW_t\\ \sigma dB_t\end{matrix}\right] \qquad &&\text{for $(R^{\eps,-}_t,\Theta^{\eps,-}_t) \notin \mathscr{S}$}\\
(R^{\eps,+}_t,\Theta^{\eps,+}_t) &= \Delta(R^{\eps,-}_t,\Theta^{\eps,-}_t) \qquad &&\text{for $(R^{\eps,-}_t,\Theta^{\eps,-}_t) \in \mathscr{S}$,}\\
(R^\eps_0,\Theta^\eps_0) &= (r_0,\theta_0) \in \BR^2\setminus\mathscr{S},
\end{aligned}
\end{equation}
where $f:\BR^2 \to \BR$ is a bounded measurable function with $\|f\|_\infty < 1$, $\sigma \in \{0,1\}$, $W_t$, $B_t$ are independent standard one-dimensional Brownian motions on some probability space $(\Omega,\filt,\BP)$, and $0 < \zeta \ll 1$ is a small parameter with $\zeta = \zeta(\eps) \searrow 0$ as $\eps \searrow 0$. Given the dependence of $\zeta$ on $\eps$, and also in an effort to keep the notation simple, we will not explicitly indicate the $\zeta$ dependence in $R^\eps_t$, $\Theta^\eps_t$. The vector field $b$ and the resetting map $\Delta$ will be assumed to satisfy the following conditions.

\begin{assumption}[Regularity of vector fields]\label{A:vf-wedge}
The function $b:\BR \to \BR$ is bounded and smooth with bounded first and second derivatives. For convenience, we define $\sfK_b\triangleq\sup_{x\in \mathbb{R}} \{|b(x)|, |b'(x)|, |b''(x)| \}$, and note that $b$ satisfies a global Lipschitz condition, with $\sfK_b$ serving as Lipschitz constant. \end{assumption}

\begin{assumption}[Resetting map]\label{A:resetting-map}
The resetting map $\Delta:\mathscr{S} \to \mathscr{R}$ is given by $\Delta(r,\alpha)=(h(r),0)$ where $h:[0,\infty) \to [0,\infty)$ is strictly increasing, differentiable, satisfies $h(0)=0$, and has continuous bounded derivatives. Thus, for every $x>0$, by the Mean Value Theorem, we have
\begin{equation}\label{E:MVT-h}
h(x)=h^\prime(\eta_x)x \quad \text{for some $\eta_x$ between $0$ and $x$.}
\end{equation}
We let $\|h^\prime\|_\infty \triangleq \sup_{r \ge 0}|h^\prime(r)|$ and for convenience,\footnote{See Propositions \ref{P:LLN-comparison}, \ref{P:CLT-comparison} and supporting lemmas.} we set
\begin{equation}\label{E:log-h-prime}
\lhp \triangleq \log(\|h^\prime\|_\infty \vee 1).
\end{equation}
\end{assumption}

For the analysis that follows, we will find it helpful to have an integral equation characterization of $(r(t),\theta(t))$ and $(R^\eps_t,\Theta^\eps_t)$ described above. To this end, we start by fixing an initial condition $(r_0,\theta_0)$ with $r_0>0$, $\theta_0=0$, and recall the form of the resetting map given by Assumption \ref{A:resetting-map}. Then, the desired solution $(r(t),\theta(t))$ to \eqref{E:sie-det} is characterized by the requirement
\begin{multline}\label{E:det-state-int-eq}
r(t) = \sum_{k \ge 1} \ind_{[t_{k-1},t_k)}(t) \left\{r^+(t_{k-1}) + \int_{t_{k-1}}^t b(r(s)) \thinspace ds \right\}, \qquad \theta(t)=\sum_{k \ge 1} \ind_{[t_{k-1},t_k)}(t) (t-t_{k-1}),\\\text{where $t_0 \triangleq 0$, $t_k \triangleq \inf\{t>t_{k-1}:(r^-(t),\theta^-(t)) \in \mathscr{S}\}=k\alpha$ for $k \ge 1$},\\   \text{$r^+(t_{k})=h(r^-(t_{k}))$ for $k \ge 1$, $r^+(t_0)=r_0$.}
\end{multline}
The assertion $t_k=k\alpha$ for $k \ge 1$ is obtained by noting that the angular velocity $\dot{\theta}$ is constant.

Similarly, the stochastic process $(R^\eps_t,\Theta^\eps_t)$ described in \eqref{E:sie-stoch} can be characterized by requiring 
\begin{multline}\label{E:stoch-state-int-eq}
R^\eps_t = \sum_{k \ge 1} \ind_{[\tau^\eps_{k-1},\tau^\eps_k)}(t) \left\{R^{\eps,+}_{\tau^\eps_{k-1}} + \int_{t\wedge \tau^\eps_{k-1}}^t b(R^\eps_s) \thinspace ds  + \eps\left(W_{t}-W_{t\wedge\tau^\eps_{k-1}} \right)  \right\},\\ \Theta^\eps_t = \sum_{k \ge 1} \ind_{[\tau^\eps_{k-1},\tau^\eps_k)}(t) \left\{\int_{t\wedge \tau^\eps_{k-1}}^t \left(1+\zeta f(R^\eps_s,\Theta^\eps_s)\right) \thinspace ds + \eps^p \sigma \left(B_t - B_{t\wedge \tau^\eps_{k-1}}\right)\right\},\\ \text{where $\tau^\eps_0 \triangleq 0$, $\tau^\eps_k \triangleq \inf\{t>\tau^\eps_{k-1}:(R^{\eps,-}_{t},\Theta^{\eps,-}_{t}) \in \mathscr{S}\}$ for $k \ge 1$,}\\ \text{$R^{\eps,+}_{\tau^\eps_k}=h(R^{\eps,-}_{\tau^\eps_k})$ for $k \ge 1$, $R^{\eps,+}_{\tau^\eps_0}=r_0$.} 
\end{multline}
A careful examination of equations \eqref{E:det-state-int-eq} and \eqref{E:stoch-state-int-eq} reveals that $(r(t),\theta(t))$ and the sample paths of the stochastic process $(R^\eps_t,\Theta^\eps_t)$ are \textit{c\`adl\`ag} functions, i.e., right-continuous for all $t \ge 0$ with left limits for $t >0$. Lemma \ref{L:existence} in Section \ref{S:Lemmas} spells out the construction of $(r(t),\theta(t))$ and $(R^\eps_t,\Theta^\eps_t)$ satisfying \eqref{E:det-state-int-eq} and \eqref{E:stoch-state-int-eq}. Note that for $n \ge 1$, the number $t_n$ and the positive random variable $\tau^\eps_n$ correspond, respectively, to the time of $n$-th impact with $\mathscr{S}$ in the deterministic and stochastic cases.

\subsection{Questions of interest and main results}\label{SS:QOI}
Our principal aim is to explore the asymptotic behavior of the stochastic process $(R^\eps_t,\Theta^\eps_t)$ and its convergence to $(r(t),\theta(t))$ in the limit as $\eps \searrow 0$ through the following questions:
\begin{itemize}
\item For a fixed time horizon $\sfT>0$, do the trajectories of the stochastic process $(R^\eps_t,\Theta^\eps_t)$ convergence in a suitable sense to $(r(t),\theta(t))$ in the limit as $\eps \searrow 0$?
\item If yes, can one obtain a more refined approximation to ($R^\eps_t, \Theta^\eps_t)$ of the form $(r(t)+\eps R^{(1)}_t, \theta(t)+ \eps \Theta^1_t)$ where the fluctuation process $(R^{(1)}_t, \Theta^1_t)$ captures the effect of the noise to leading order?
\end{itemize}
Since the stochastic processes in question have \textit{c\`adl\`ag} paths, issues of convergence over the time horizon $[0,\sfT]$ with $\sfT>0$ are best studied in the space $D([0,\sfT];\BR^2)$ of functions from $[0,\sfT]$ into $\BR^2$ which are right-continuous with left limits, equipped with the Skorohod topology \cite{ConvProbMeas, EK86}. To describe the metric and topological structure of this function space, we briefly consider a slightly more general setting.

Let $(S,\rho)$ be a complete separable metric space, and for $\sfT>0$, let $D([0,\sfT];S)$ denote the space of all functions $x:[0,\sfT] \to S$ which are right-continuous for all $t \ge 0$ and possess left limits for $t>0$. We start by defining a preliminary family of time distortions $\tilde{\Lambda}_\sfT$ consisting of all strictly increasing continuous bijections on $[0,\sfT]$. Next, we let $\Lambda_\sfT$ be the family of all $\lambda \in \tilde{\Lambda}_\sfT$ for which
\begin{equation}\label{E:time-distortion-cost}
\gamma_\sfT(\lambda) \triangleq \sup_{0 \le s <t \le \sfT} \left|\log \frac{\lambda(t)-\lambda(s)}{t-s}\right| <\infty.
\end{equation}
Note that the identity function $\itd(t) \equiv t$ belongs to $\Lambda_\sfT$, and we have $\gamma_\sfT(\itd)=0$. The \textit{Skorohod metric} $d_{[0,\sfT]}$ on $D([0,\sfT];S)$ is now defined by setting 
\begin{equation}\label{E:Sk-metric}
d_{[0,\sfT]}(x_1,x_2) \triangleq \inf_{\lambda \in \Lambda_\sfT}\left\{ \gamma_\sfT(\lambda) \vee \sup_{0 \le t \le \sfT} \rho\left(x_1(t),x_2(\lambda(t))\right) \right\}, \qquad \text{for $x_1$, $x_2 \in D([0,\sfT];S)$.}
\end{equation}
The intuition behind the Skorohod metric is that two functions $x_1$ and $x_2$ in $D([0,\sfT];S)$ should be close if, after allowing a small distortion of time (with ``smallness" captured by $\gamma_\sfT(\lambda)$), the trajectories of $x_1(\cdot)$ and $x_2(\lambda(\cdot))$ are uniformly close on $[0,\sfT]$. We also note that if $d^u_{[0,\sfT]}(x_1,x_2) \triangleq \sup_{0 \le t \le \sfT} \rho(x_1(t),x_2(t))$ denotes the uniform metric on $D([0,\sfT];S)$, then $d_{[0,\sfT]}(x_1,x_2) \le d^u_{[0,\sfT]}(x_1,x_2)$ since the latter corresponds to the expression inside the infimum in \eqref{E:Sk-metric} for the specific choice $\lambda = \itd$.

\begin{definition}[Time horizon]
Fix $\sfN \in \BN$. The time horizon for our problem will be $[0,\sfT]$ with $\sfT\in (\sfN \alpha, (\sfN+1)\alpha)$.
\end{definition}

Let $x(t) \triangleq (r(t),\theta(t))$, $X^\eps_t=(R^\eps_t,\Theta^\eps_t)$. Our first main result, stated as Theorem \ref{T:LLN} below, asserts that $d_{[0,\sfT]}(X^\eps,x) \to 0$ in $L^\beta$ as $\eps \searrow 0$.

\begin{theorem}[Law of large numbers]\label{T:LLN}
Suppose that $f \equiv 0$ and $\sigma=1$. 
 Let $x(t)$ and  $X^\varepsilon_t$ solves the systems $\eqref{E:sie-det}$ and $\eqref{E:sie-stoch}$ respectively, and fix $\beta \in \{1,2\}$. Fix $p>1$ and $\sfT>0$. For given $\eps\in(0,1)$, sufficiently small, there exists a constant $C>0$ such that, \begin{equation*}
     \mathbb{E}\left[d_{[0,\sfT]}(X^\eps,x)^\beta\right]\le C \eps^\beta
 \end{equation*}
\end{theorem}
 
Let $Z_t \triangleq (R^1_t,\Theta^1_t)$ be the process solving 
\begin{equation}\label{E:fluct-proc}
\begin{aligned}
R^1_t &\triangleq \sum_{k \ge 1} \ind_{[t_{k-1},t_k)}(t) \left\{R^{1,+}_{t_{k-1}} + \int_{t_{k-1}}^t b^\prime(r(s)) R^1_s \thinspace ds +(W_t-W_{t_{k-1}}) \right\},\\
\Theta^1_t &\triangleq 0 
\\
\text{where }
R^{1,+}_{t_k}&=h^\prime(r^-(t_k)) R^{1,-}_{t_k} \quad \text{for $k \ge 1$}, \qquad R^{1,+}_0 \triangleq 0, \qquad t_k=k\alpha \quad \text{for $k \in \BZ^+$}.
\end{aligned}
\end{equation}

\begin{remark}\label{R:fluct-proc}
Note that the process $(R^1_t,\Theta^1_t)$ is c\`adl\`ag with jumps at the deterministic impact times $t_k=k\alpha$, $k \ge 1$. In between jumps, the process $(R^1_t,\Theta^1_t)$ solves the {\sc sde} $dR^1_t = b^{\prime}(r(t))R^1_t \ dt+ dW_t$, $d\Theta^1_t =0$ with resetting condition $R^{1,+}_{t_k}=h^\prime(r^-(t_k)) R^{1,-}_{t_k}$, $\Theta^{1,+}_{t_k} = 0$. One observes that both the {\sc sde} and the resetting condition for $R^1_t$ are non-autonomous (on account of dependence on $r(t)$) and involve a \textit{linearization} about the deterministic trajectory $r(t)$ of the (potentially nonlinear) vector field $b$ and resetting condition $x \mapsto h(x)$. 
\end{remark}

\color{black} Theorem \ref{T:LLN} ensures that for any fixed $\sfT>0$, the quantity $d(X^\varepsilon, x),$ converges to $0$ in $L^p$. Given that $d(X^\varepsilon, x)$ may be thought of as being of order $\mathcal{O}(\varepsilon)$, it is natural to ask whether one can identify any of the higher order terms in an expansion of $X^\varepsilon_t$ in powers of the small parameter$\varepsilon$. We would like to identify a stochastic process $Z_t$, independent of $\varepsilon$, such that $X^\varepsilon_t=x_t+\varepsilon Z_t+o(\varepsilon)$, where we can explicitly determine the error $d(X^\varepsilon, x+\varepsilon Z). $\color{black}

\begin{theorem}[Central limit theorem]\label{T:CLT}
Suppose that $f \equiv 0$ and $\sigma=1$. 
Let $x(t)$ and  $X^\varepsilon_t$ solves the systems $\eqref{E:sie-det}$ and $\eqref{E:sie-stoch}$ respectively. Fix $p>1$ and $\sfT>0$. For given $\eps\in(0,1)$, sufficiently small and $1<\nu<p$, fixed, there exists a constant $C>0$ such that,
 \begin{equation*}
     \mathbb{E}\left[d_{[0,\sfT]}(X^\eps,x+\eps Z)\right]\le\begin{cases}
     	& C \eps^\nu, \quad \text{when $\nu< p\le 2$}\\
     	& C \eps^2, \quad \text{when $2<\nu<p$}
     \end{cases} 
 \end{equation*}
\end{theorem}

\color{black}Before proceeding with the proofs of Theorems \ref{T:LLN} and \ref{T:CLT}, let us first outline our strategy. Recalling the discussion about the function space $D([0,\sfT];\BR^2)$ in Section \ref{S:ProblemStatement}, we see that a sample path of $X^\eps_t = (R^\eps_t,\Theta^\eps_t)$ should be close to $x(t) = (r(t),\theta(t))$, or $x(t) + \eps Z_t$ for that matter, in $D([0,\sfT];\BR^2)$ if the impact times in the stochastic case are close to their deterministic counterparts, and after allowing this small mismatch in times, the state trajectories follow roughly the same ``spatial" course.
To this end, we will compare the sample paths of $X^\eps_t$ with $x(t)$ or $x(t) + \eps Z_t$ using piecewise linear time distortions which \textit{align} the jumps in the stochastic and deterministic cases with high probability, obtaining in the process upper bounds for the true Skorohod distance between $(R^\eps_t,\Theta^\eps_t)$ and $(r(t),\theta(t))$. Showing that these upper bounds are small in an $L^p$ sense for $0<\eps \ll 1$ will follow from using the smallness of the noise, the continuity of the resetting map, and showing closeness of $\tau^\eps_n$ and $t_n$ with high probability. While the exact form of the time distortion we will use will depend on the nature of the noise in the angular coordinate $\Theta^\eps_t$, it will always be of the form given below, with suitable $G^\eps_\sfN$.

\begin{definition}[Random time distortion]\label{D:rtd}
For $\eps \in (0,1)$, let $G_\sfN^\eps \in \filt$ be an event satisfying, among perhaps other assumptions, the requirement that for all $\omega \in G_\sfN^\eps$, we have $|\tau^\eps_n(\omega)-t_n| < \alpha/4$ for all $0 \le n \le \sfN$
. Define now the random time distortion $\lambda^\eps$ by
\begin{equation}\label{E:rtd-generic}
\lambda^\eps_t(\omega) \triangleq \begin{cases}\sum_{k=1}^{\sfN} \ind_{[t_{k-1}, t_{k})} (t) \left\{\tau^{\varepsilon}_{k-1} (\omega) + \left(\frac{\tau^\varepsilon_{k}(\omega)-\tau^{\varepsilon}_{k-1} (\omega)}{t_{k} - t_{k-1}}\right) (t- t_{k-1}) \right\}\\ \qquad  + \ind_{[t_{\sfN}, \sfT]}(t) \left\{\tau^\varepsilon_{\sfN}(\omega) + \left(\frac{\sfT-\tau^\varepsilon_{\sfN}(\omega)}{\sfT - t_{\sfN}}\right) (t- t_{\sfN})\right\} \qquad &\text{for $t \in [0,\sfT]$, $\omega \in G^\eps_\sfN$}\\t &\text{for $t \in [0,\sfT]$, $\omega \in \Omega\setminus G^\eps_\sfN$.}\end{cases}
\end{equation}
\end{definition}

Noting that each $t \mapsto \lambda^\eps_t(\omega)$ is differentiable a.e. and recalling that $t_k-t_{k-1}=\alpha$ for all $1\le k \le \sfN$, let $J^\eps$ be the function
\begin{equation}\label{E:J-generic}
J^\eps_{t}(\omega) \triangleq \begin{cases}\sum_{k=1}^{\sfN} \ind_{[t_{k-1}, t_{k})} (t) \left(\frac{\tau^\varepsilon_{k}(\omega)-\tau^{\varepsilon}_{k-1} (\omega)}{\alpha}\right)  + \ind_{[t_{\sfN}, \sfT]}(t) \left(\frac{\sfT-\tau^\varepsilon_{\sfN}(\omega)}{\sfT-\sfN\alpha}\right) \quad &\text{for $t \in [0,\sfT]$, $\omega \in G^\eps_\sfN$}\\1 &\text{for $t \in [0,\sfT]$, $\omega \in \Omega\setminus G^\eps_\sfN$}\end{cases}
\end{equation}
which equals $\frac{d}{dt}\lambda^\eps_t(\omega)$ for a.e. $t \in [0,\sfT]$, for all $\omega \in \Omega$. Note also that for $\omega \in G^\eps_\sfN$, the time deformation $\lambda^\eps_\cdot(\omega)$ aligns all impact times in $[0, \sfT]$, i.e., $\lambda^\eps_{t_k}(\omega)=\tau^\eps_{k}(\omega)$ for $0 \le k \le \sfN$. The requirement of the time distortion being a bijection on $[0,\sfT]$ forces $\lambda^\eps_{\sfT}(\omega)=\sfT$.
Before proceeding with the detailed calculations, it will be helpful to introduce quantities $\ren(t)$ and $\Ren^\eps_t$ which count the total number of impacts up to (and including) time $t \in [0,\infty)$ in the deterministic and stochastic settings. In the former case, the impacts occur at times $t_n=n\alpha$, $n \in \BZ^+$, while in the latter case the impact times are given by the increasing sequence of positive random variables $\{\tau_n^\eps\}_{n=0}^\infty$. Thus, the total number of impacts that have occurred by time $t$ in these two cases are given, respectively, by the function $\ren(t)$  and the process $\Ren^\eps_t$ defined by 
\begin{equation}\label{E:renRen}
\ren(t) = \lfloor \frac{t}{\alpha} \rfloor, \qquad \Ren^\eps_t \triangleq \max\{k \in \BZ^+:\tau^\eps_k \le t\}.
\end{equation}
Owing to the fact that our system starts ``afresh" after every impact, the times \textit{between} impacts, viz., $\{\tau^\eps_n - \tau^\eps_{n-1}\}_{n=1}^\infty$, form a family of independent identically distributed random variables, thereby making $\Ren^\eps_t$ a renewal process. Our next result, Lemma \ref{L:FPT}, provides the density of $\tau^\eps_n-\tau^\eps_{n-1}$, and perhaps more importantly, provides upper bounds on the quantity $\BP\left(|\tau^\eps_n - \tau^\eps_{n-1} - \alpha| \ge \delta\right)$ for $\delta \in (0,1)$ small enough.

\begin{lemma}\label{L:FPT}
Consider the case when $f \equiv 0$ and $\sigma=1$. Then, the density $f^\eps(t)$ of $\tau^\eps_n-\tau^\eps_{n-1}$ for each $n \ge 1$ (and hence of $\tau^\eps_1$) is given by
\begin{equation}\label{E:FPT-pdf}
\BP\left(\tau^\eps_n - \tau^\eps_{n-1} \in dt\right) = f^\eps(t) \thinspace dt \triangleq \frac{\alpha}{\eps^p \sqrt{2\pi t^3}} \exp\left[-\frac{1}{2\eps^{2p} t}(t-\alpha)^2\right] \thinspace dt, \qquad \text{for $t>0$.}
\end{equation}
Further, there exists $K_{\ref{L:FPT}}>0$ such that for all $0 < \delta < \min\{1,\alpha/2\}$, we have
\begin{equation}\label{E:FPT-estimates}
\BP\left(|\tau^\eps_n - \tau^\eps_{n-1} - \alpha| \ge \delta\right) \le K_{\ref{L:FPT}}\frac{\eps^p}{\delta} \exp\left(-\frac{\delta^2}{4\alpha\eps^{2p}}\right).
\end{equation}
\end{lemma}

To prove this lemma, we will use some results about inverse Gaussian random variables. 
\begin{proposition}[\cite{shuster1968inverse}]\label{P:CDF-IG}\,
Suppose $X$ is a positive random variable with an inverse Gaussian distribution with parameters $\hat{\mu} > 0$, $\lambda > 0$, i.e., $X$ has density
\begin{equation*}
f(x)= \begin{cases}
		\left( \frac{\lambda}{2 \pi x^3} \right)^{\frac{1}{2}} \exp \left[ -\frac{\lambda (x - \hat{\mu})^2}{2 \hat{\mu}^2 x} \right], &\quad x > 0,\\
		0, & \quad x \leq 0.
	\end{cases} 
\end{equation*}
Let $F(c; \hat{\mu}, \lambda) = \BP(X \le c)$ denote the cumulative distribution function of $X$, set
\begin{equation*}
a \triangleq \frac{\lambda (c - \hat{\mu})^2}{\hat{\mu}^2 c} \quad \text{and for $z>0$, let} \quad \quad G(z) \triangleq \int_z^\infty (2 \pi t)^{-\frac{1}{2}} \exp(-t/2) \, dt 
\end{equation*}
be the tail probability of $\chi^2_1$.
Then	
\begin{itemize}
\item[(A)] $F(c; \hat{\mu}, \lambda) = \frac{1}{2} G(a) + \frac{1}{2} \exp(2 \lambda / \hat{\mu}) G(a + 4 \lambda / \hat{\mu}), \quad 0 < c \leq \hat\mu;$
\item[(B)] $F(c; \hat{\mu}, \lambda) = 1 - \frac{1}{2} G(a) + \frac{1}{2} \exp(2 \lambda / \hat{\mu}) G(a + 4 \lambda / \hat{\mu}), \quad\hat{\mu} < c < \infty,$
\end{itemize}	
\end{proposition}

We can now provide the proof of the above lemma.

\begin{proof}[Proof of Lemma \ref{L:FPT}] Suppose that $f \equiv 0$, $\sigma=1$. We start by noting that for any $n \ge 1$, the random variable $\tau^\eps_n-\tau^\eps_{n-1}$ has the same distribution as the first passage time ({\sc fpt}) to the level $\alpha>0$ for the process $t+\eps^p B_t$, or equivalently, the {\sc fpt} to level $\alpha/\eps^p$ for the process $B_t + t/\eps^p$. Given that the latter is simply a Brownian motion with drift, it follows from \cite[Section 3.5C]{KS91} that $\tau^\eps_n - \tau^\eps_{n-1}$ has density given by equation \eqref{E:FPT-pdf}. Comparing with Proposition \ref{P:CDF-IG}, we see that $\tau^\eps_n - \tau^\eps_{n-1}$ has an inverse Gaussian distribution with $\lambda=\frac{\alpha^2}{\eps^{2p}}$ and $\hat{\mu}= \alpha$. 

To prove \eqref{E:FPT-estimates}, we start by noting that for $0 < \delta < \min\{1,\alpha/2\}$, $\BP\left(|\tau^\eps_n - \tau^\eps_{n-1} - \alpha| \ge \delta\right) = \BP\left(\tau^\eps_n - \tau^\eps_{n-1} \le \alpha-\delta\right) +\BP\left(\tau^\eps_n - \tau^\eps_{n-1} \ge \alpha + \delta\right)$. We will also employ the inequality $G(z) \le \frac{2}{\sqrt{2\pi z}}e^{-z/2}$ for $z>0$, which follows by noting that $G(z) \le (2\pi z)^{-1/2}\int_z^\infty e^{-t/2} \thinspace dt$ and computing the integral. We start with $\BP\left(\tau^\eps_n - \tau^\eps_{n-1} \le \alpha-\delta\right)$, noting that by Proposition \ref{P:CDF-IG}, we have
$\BP\left(\tau^\eps_n - \tau^\eps_{n-1} \le \alpha-\delta\right)=\frac{1}{2}G\left(\frac{\delta^2}{\eps^{2p}(\alpha-\delta)}\right) + \frac{1}{2}e^{2\alpha/\eps^{2p}} G\left(\frac{\delta^2}{\eps^{2p}(\alpha-\delta)}+\frac{4\alpha}{\eps^{2p}}\right)$.
Some straightforward calculations using the inequality above for $G(z)$, the monotone nature of various functions (e.g., $x \mapsto 1/\sqrt{x}$, $x \mapsto e^{-x}$) together with the inequalities $0<\alpha/2 < \alpha-\delta < \alpha < \alpha+\delta < 2\alpha$ yield
\begin{equation}\label{E:FPT-Estimate-1}
\BP\left(\tau^\eps_n - \tau^\eps_{n-1} \le \alpha-\delta\right) \le \frac{\eps^p}{\delta}\sqrt{\frac{\alpha}{2\pi}} e^{-\delta^2/2\eps^{2p}\alpha} + \frac{\eps^p}{\sqrt{8\pi \alpha}}e^{-\delta^2/2\eps^{2p}\alpha}
\end{equation}
Next, we note that $\BP\left(\tau^\eps_n - \tau^\eps_{n-1} \ge \alpha + \delta\right)=\frac{1}{2}G\left(\frac{\delta^2}{\eps^{2p}(\alpha+\delta)}\right) - \frac{e^{2\alpha/\eps^{2p}}}{2}G\left(\frac{\delta^2}{\eps^{2p}(\alpha+\delta)} + \frac{4\alpha}{\eps^{2p}}\right) \le \frac{1}{2}G\left(\frac{\delta^2}{\eps^{2p}(\alpha+\delta)}\right)$. Once again, using the estimate for $G$ along with some elementary calculations based on the observations above, we obtain
\begin{equation}\label{E:FPT-Estimate-2}
\BP\left(\tau^\eps_n - \tau^\eps_{n-1} \ge \alpha + \delta\right) \le \frac{\eps^p}{\delta}\sqrt{\frac{\alpha}{\pi}} e^{-\delta^2/4\eps^{2p}\alpha}.
\end{equation}
Combining equations \eqref{E:FPT-Estimate-1}, \eqref{E:FPT-Estimate-2} and noting that $\eps^p \le \eps^p/\delta$, we now obtain \eqref{E:FPT-estimates}.
\end{proof}


\subsubsection{\textbf{Probability of the event that stochastic and deterministic times are closed :}\nopunct}\		

For $\varepsilon, \delta\in(0,1)$, set $G^{\varepsilon, \delta}_0\triangleq \Omega$ and for $n\geq 1$, define
$$
\begin{aligned}
	& G^{\varepsilon,\delta}_n \triangleq \{\omega\in G^{\varepsilon,\delta}_{n-1} :  |\tau^{\varepsilon}_n - t_n|\leq\delta\} \\
	& B^{\varepsilon,\delta}_n \triangleq \{\omega\in G^{\varepsilon,\delta}_{n-1} :  |\tau^{\varepsilon}_n - t_n|>\delta\}= G^{\varepsilon,\delta}_{n-1}\setminus G^{\varepsilon,\delta}_n
\end{aligned}
$$
Consequently, for $\sfT\in\mathbb{N}$,  $\bigcap_{n=1}^{\sfT } G^{\varepsilon,\delta}_n = G^{\varepsilon,\delta}_{\sfT} $ ,
Then, $\mathbb{P}(B^{\varepsilon,\delta}_{\sfT})= \mathbb{P}(\bigcup_{n=1}^{\sfT} B^{\varepsilon,\delta}_n)$  . \\
\noindent
\begin{lemma}\label{L:prob-B}
	For any $\sfT\in \BN$, there exists $C_{\ref{L:prob-B}}>0$ $$\mathbb{P}(B^{\varepsilon,\delta}_{\sfT})\le C_{\ref{L:prob-B}} \frac{\eps^p}{\delta} \exp\left(-\frac{\delta^2}{4\alpha\eps^{2p}}\right) $$
\end{lemma}

\begin{proof}[Proof of Lemma \ref{L:prob-B}]
For any $n\ge 1$ using the Lemma \ref{L:FPT}
\begin{multline*}
	\mathbb{P}(B^{\eps, \delta}_n)
	\le \sum_{j=1}^{n} \mathbb{P}(|\tau^\eps_j - \tau^\eps_{j-1}-\alpha|\ge \frac{\delta}{n})
	=n\mathbb{P}(|\tau^\eps_n - \tau^\eps_{n-1}-\alpha|\ge \frac{\delta}{n})
	\le n K_{\ref{L:FPT}} \frac{\eps^p}{\delta} \exp\left(-\frac{\delta^2}{4\alpha\eps^{2p}}\right)
\end{multline*}
However, $\mathbb{P}(B^{\varepsilon,\delta}_{\sfT})= \mathbb{P}(\bigcup_{n=1}^{\sfT} B^{\varepsilon,\delta}_n)\le \sum_{n=1}^{\sfT}  \mathbb{P}(B^{\varepsilon,\delta}_n)$ . 
Therefore, 
\begin{equation*}
	\begin{aligned}
		\sum_{n=1}^{\sfT} \mathbb{P}(B^{\varepsilon,\delta}_n)
		\le   K_{\ref{L:FPT}} \sfT (\sfT+1) \frac{\eps^p}{\delta} \exp\left(-\frac{\delta^2}{4\alpha\eps^{2p}}\right)\quad , 
	\end{aligned}
\end{equation*}
which completes the proof of the stated lemma. 
\end{proof}

\color{black} 

\begin{lemma}\label{L:distorion}
Let $\sfT\in \mathbb{N}$. For $0<\delta<\frac{\alpha}{4\sfT}$, 	\begin{equation}
					\gamma_{\sfT} (\lambda^{\varepsilon}(\omega)):=\sup_{0\leq s<t\leq \sfT} \left|\log \left(\frac{\lambda^{\varepsilon}_t (\omega) -\lambda^\varepsilon_s(\omega)}{t-s}\right)\right|\leq \frac{4\sfT\delta}{\alpha}   
			\end{equation} 
\end{lemma}
\begin{proof}[Proof of Lemma \ref{L:distorion}]
	Clearly, for $\omega\in\Omega\setminus G^{\varepsilon,\delta}_\sfT,\gamma_{\sfT} (\lambda^\varepsilon(\omega))=0.$\newline
	Fix $\omega\in G^{\varepsilon,\delta}_\sfT$. Then function $\lambda^\varepsilon_{.}(\omega)$ is piecewise linear with corners at $\tau^\varepsilon_k, \sfT; k=0,1,2,\ldots,n$\newline
	Since, $\omega\in G^{\varepsilon,\delta}_{\sfT}$, $\tau^\varepsilon(\omega)\leq [t_n-\delta, t_n+\delta]$, for $1\leq n\leq \sfT$ and $\lambda^\varepsilon_{t_n}(\omega)=\tau^\varepsilon(\omega)$.
	\begin{equation*}
		\max_{1\leq n\leq \sfT}\left|\frac{\lambda^\varepsilon_{t_n}(\omega)-\lambda^\varepsilon_{t_{n-1}}(\omega)}{t_n-t_{n-1}}-1\right|\leq \frac{2\delta}{\alpha}
	\end{equation*}
	Let $0\leq s<t<\sfT$ and let $u_0,u_1,\ldots,u_n$ be corners just to the left of $s$ ending to the right of $t$ i.e.
	\begin{equation*}
		\begin{aligned}
			\left| \lambda^\varepsilon_{t}(\omega)-\lambda^\varepsilon_{s}(\omega)-(t-s)\right|&\leq |t-u_{k-1}|\left|\frac{\lambda^\varepsilon_{t}(\omega)-\lambda^{\varepsilon}_{u_{k-1}}(\omega)}{t-u_{k-1}}-1\right|+ \sum_{i=2}^{k-1} |u_i-u_{i-1}|\left|\frac{\lambda^{\varepsilon}_{u_i}(\omega)-\lambda^{\varepsilon}_{u_{i-1}}(\omega)}{u_i-u_{i-1}}-1\right|\\
			&\quad + |u_1 - s| \left|\frac{\lambda^{\varepsilon}_{u_1}(\omega)- \lambda^\varepsilon_s(\omega)}{u_1 - s}-1\right|
		\end{aligned}
	\end{equation*}
		Therefore, 
		\begin{equation*}
			\begin{aligned}
				\left|\frac{ \lambda^\varepsilon_{t}(\omega)-\lambda^\varepsilon_{s}(\omega)}{t-s}-1\right|&\leq \left|\frac{t-u_{k-1}}{t-s}\right|\left|\frac{\lambda^\varepsilon_{t}(\omega)-\lambda^{\varepsilon}_{u_{k-1}}(\omega)}{t-u_{k-1}}-1\right|+ \sum_{i=2}^{k-1} \left|\frac{u_i-u_{i-1}}{t-s}\right|\left|\frac{\lambda^{\varepsilon}_{u_i}(\omega)-\lambda^{\varepsilon}_{u_{i-1}}(\omega)}{u_i-u_{i-1}}-1\right|\\
				&\quad + \left|\frac{u_1 - s}{t-s}\right| \left|\frac{\lambda^{\varepsilon}_{u_1}(\omega)- \lambda^\varepsilon_s(\omega)}{u_1 - s}-1\right|
			\end{aligned}
		\end{equation*}
		This implies 
		\begin{equation*}
			\begin{aligned}
				\left|\frac{ \lambda^\varepsilon_{t}(\omega)-\lambda^\varepsilon_{s}(\omega)}{t-s}-1\right|\leq \sum_{i=1}^{k} \left|\frac{\lambda^{\varepsilon}_{u_i}-\lambda^{\varepsilon}_{u_{i-1}}} {u_i-u_{i-1}}-1\right|\leq \frac{2\sfT}{\delta}
			\end{aligned}
		\end{equation*}
		Finally, for $0<\delta\leq \frac{\alpha}{4\sfT}$, 
		\begin{equation}
			\begin{aligned}
				\log \left(1-\frac{2T\delta}{\alpha}\right)
				&\leq \log \left(\frac{ \lambda^\varepsilon_{t}(\omega)-\lambda^\varepsilon_{s}(\omega)}{t-s}\right)\leq\log \left(1+\frac{2\sfT\delta}{\alpha}\right)\\
				-\frac{4\sfT\delta}{\alpha}
				&\leq \log \left(\frac{ \lambda^\varepsilon_{t}(\omega)-\lambda^\varepsilon_{s}(\omega)}{t-s}\right)\leq \frac{4\sfT\delta}{\alpha}\\
				&\gamma_{T}(\lambda^\varepsilon(\omega))\leq \frac{4\sfT\delta}{\alpha}
			\end{aligned}
		\end{equation}
	\end{proof}

\begin{lemma}\label{L:Jacobian-Estimates}
	Given $\delta>0$ same as the in set $G^\eps_{\cdot}$, the following holds,
	\begin{equation*}
		\|J^\eps_{\cdot} -1\|_\infty \le  \frac{2\delta}{\alpha}\wedge \frac{\delta}{\sfT-\ren(\sfT)\alpha}
	\end{equation*}
\end{lemma}

\begin{proof}[Proof of Lemma \ref{L:Jacobian-Estimates}]
	For $t\in [0, \sfT]$, 
	\begin{equation*}
		J^\eps_{t}(\omega) \triangleq \begin{cases}\sum_{k=1}^{\sfN} \ind_{[t_{k-1}, t_{k})} (t) \left(\frac{\tau^\varepsilon_{k}(\omega)-\tau^{\varepsilon}_{k-1} (\omega)}{\alpha}\right)  + \ind_{[t_{\sfN}, \sfT]}(t) \left(\frac{\sfT-\tau^\varepsilon_{\sfN}(\omega)}{\sfT-\sfN\alpha}\right) \quad &\text{for $\omega \in G^\eps_\sfN$}\\1 &\text{for $\omega \in \Omega\setminus G^\eps_\sfN$}\end{cases}
	\end{equation*}
	
	Take any $t\in [0, \sfT]$, then $\exists\, k\in \{1,\dots, \sfN\}$ such that $t\in [t_{k-1}, t_k)\cup[t_{\sfN}, \sfT]$

	$$|J^\eps_{t}(\omega) -1| \le \left|\left(\frac{\tau^\varepsilon_{k}(\omega)-\tau^{\varepsilon}_{k-1} (\omega)}{\alpha}\right) -1\right| \vee \left|\left(\frac{\sfT-\tau^\varepsilon_{\sfN}(\omega)}{\sfT -\sfN\alpha}\right) -1 \right|$$
\noindent
	However, for $\omega\in G^\eps_\sfN$,  $|\tau^\eps_k(\omega) - k\alpha|\le \delta \text{ which implies } |\tau^\eps_k(\omega) - \tau^\eps_{k-1}(\omega) -\alpha|\le 2\delta$ and $\left|\left(\frac{\sfT-\tau^\varepsilon_{\sfN}(\omega)}{\sfT-\sfN\alpha}\right) -1 \right|= \left|\frac{\sfN\alpha-\tau^\eps_\sfN(\omega)}{\sfT-\sfN\alpha}\right|\le \frac{\delta}{\sfT-\ren(\sfT)\alpha} $\, .

	Hence, $|J^\eps_{t}(\omega) -1|\le \frac{2\delta}{\alpha}\wedge \frac{\delta}{\sfT-\ren(\sfT)\alpha}$, for any $t\in [0, \sfT]$. Consequently, $\|J^\eps_\cdot -1 \|_\infty \le \frac{2\delta}{\alpha}\wedge \frac{\delta}{\sfT-\ren(\sfT)\alpha}$\, .
\end{proof}

In essence, we are claiming that $\|J^\eps_\cdot -1 \|_\infty$ is bounded above by $C \delta$, for some constant $C>0$. 
\color{black} 

\begin{lemma}\label{L:Renewal-Process}
	Let $t\in [0, \sfT]$. For a given fixed $\gamma>0$, there exists $\lambda>0$ with  $\lambda> \frac{\gamma}{\alpha}+ \frac{\gamma^2}{2\alpha^2}$
	 such that
	\begin{equation*}
		\begin{aligned}
			\mathbb{E}\left[e^{\gamma Q^\eps_t}\right]&\le  \frac{e^{\lambda t}}{1- e^\gamma \mathbb{E}\left[ e^{-\lambda \tau^\eps_1}\right]}, \quad \text{provided } e^\gamma \mathbb{E}\left[ e^{-\lambda \tau^\eps_1}\right]<1\\
			& \text{where, } \mathbb{E}\left[e^{-\lambda \tau^\eps_1}\right]= \exp\left(\frac{\alpha}{\eps^{2p}} \left(1-\sqrt{1+2\lambda \eps^{2p}}\right)\right); \quad (\text{cf. \cite{KS91}})
		\end{aligned}
	\end{equation*}
	In fact, with the assumptions in lemma \ref{L:Renewal-Process}, there further exists a constant $\kappa(\lambda)$ independent of the parameter $\eps$ such that $$
	\mathbb{E}\left[e^{\gamma Q^\eps_t}\right]\le \kappa(\lambda) e^{\lambda t}
	$$
\end{lemma}

\begin{proof}[Proof of Lemma \ref{L:Renewal-Process}]
	We begin the proof by first noticing $\{Q^\eps_t\ge n\}\iff \{\tau^\eps_n \le t\}$.
	
	\begin{multline*}
		\mathbb{E}\left[e^{\gamma Q^\eps_t}\right]
		= \sum_{n\ge 0} e^{\gamma n} \mathbb{P}(Q^\eps_t)=n)
		\le \sum_{n\ge 0} e^{\gamma n} \mathbb{P}(Q^\eps_t)\ge n)\le  \sum_{n\ge 0} e^{\gamma n} \mathbb{P}(\tau^\eps_n \le t)  \\
		\text{Using Chernoff's inequality, for any } \lambda>0 \\
		\le \sum_{n\ge 0} e^{\gamma n} e^{\lambda t} \mathbb{E}\left[e^{-\lambda \tau^\eps_n}\right] = \sum_{n\ge 0} e^{\gamma n } e^{\lambda t} \mathbb{E}\left[e^{-\lambda (\tau^\eps_n - \tau^\eps_{n-1}+\tau^\eps_{n-1}-\tau^\eps_{n-2}+\ldots+\tau^\eps_1-\tau^\eps_0)}\right]=  \sum_{n\ge 0} e^{\gamma n } e^{\lambda t} \mathbb{E}\left[e^{-\lambda \tau^\eps_1}\right]^n\\
		\text{where we have used the fact that the sequence} \{\tau^\eps_n - \tau^\eps_{n-1}\}_{n\ge 1} \text{is iid \cite[Section 3.5C]{KS91}}
	\end{multline*}
	
Therefore, $$
	\mathbb{E}\left[e^{\gamma Q^\eps_t}\right]\le \frac{e^{\lambda t}}{1- e^\gamma \mathbb{E}\left[ e^{-\lambda \tau^\eps_1}\right]}, \quad \text{provided } e^\gamma \mathbb{E}\left[ e^{-\lambda \tau^\eps_1}\right]<1
$$

However, from \cite{KS91},  we know the fact that for $\lambda>0$, 
$$ \mathbb{E}\left[e^{-\lambda \tau^\eps_1}\right]= \exp\left(\frac{\alpha}{\eps^{2p}} \left(1-\sqrt{1+2\lambda \eps^{2p}}\right)\right)$$

Therefore, we must have $\gamma+\frac{\alpha}{\eps^{2p}} \left(1-\sqrt{1+2\lambda \eps^{2p}}\right)<0 $ which implies $\lambda> \frac{\gamma}{\alpha}+\frac{\gamma^2}{2\alpha^2}\eps^{2p}$.
	
We next proceed to find a lower bound for the denominator. $1- e^\gamma \mathbb{E}\left[ e^{-\lambda \tau^\eps_1}\right]$.

Using the expression $$ \mathbb{E}\left[e^{-\lambda \tau^\eps_1}\right]= \exp\left(\frac{\alpha}{\eps^{2p}} \left(1-\sqrt{1+2\lambda \eps^{2p}}\right)\right)$$ we write
\begin{equation*}
	\begin{aligned}
		1- e^\gamma \mathbb{E}\left[ e^{-\lambda \tau^\eps_1}\right]= 1- \exp\left(\frac{\alpha}{\eps^{2p}} \left(1-\sqrt{1+2\lambda \eps^{2p}}\right)\right)
		= 1- exp\left(\gamma - \frac{2\alpha\lambda}{1+\sqrt{1+2\lambda \eps^{2p}}}\right)
	\end{aligned}
\end{equation*}

Now with $0<\eps<1$, 
\begin{equation*}
	\begin{aligned}
		1- exp\left(\gamma - \frac{2\alpha\lambda}{1+\sqrt{1+2\lambda \eps^{2p}}}\right)& \ge 1+ \left( \frac{2\alpha\lambda}{1+\sqrt{1+2\lambda \eps^{2p}}}-\gamma\right)^{-1}\\
		&\ge 1+ \left( \frac{2\alpha\lambda - \gamma\left(1+\sqrt{1+2\lambda }\right)}{1+\sqrt{1+2\lambda }}\right)^{-1}\triangleq \kappa(\lambda)^{-1}
	\end{aligned}
\end{equation*}
 Since, $\eps\in(0, 1)$, one may choose  $\lambda>\frac{\gamma}{\alpha}+\frac{\gamma^2}{2\alpha^2}$ so that the assumptions continue to hold and Lemma \ref{L:Renewal-Process} remains valid.
\end{proof}


\section{Limiting Mean Behavior}\label{S:LLN}
In the present section, we provide the proof of Theorem \ref{T:LLN}, breaking the arguments into a series of lemmas. We now state a couple of lemmas which enable us to estimate the quantity $|R^\eps_{\lambda^\eps_t}-r(t)|$ and $|\Theta^\eps_{\lambda^\eps_t}-\theta(t)|$ on $G^\eps_\sfN$ and $\Omega\setminus G^\eps_\sfN$. The proofs are given in Section \ref{S:Lemmas}. 
\begin{proposition}\label{P:LLN-comparison}
There exists a constant $C=C(\sfT)>0$ such that for $\eps \in (0,1)$, we have
\begin{equation*}
\begin{aligned}
|R^\eps_{\lambda^\eps_t}-r(t)| &\le C\left\{\|J^\eps_\cdot-1\|_\infty + \eps \sup_{s \in [0,\sfT]} |W_s|\right\} \ind_{G^\eps_\sfN} + Ce^{\varrho \Ren^\eps_t}\left\{1 + \eps (\Ren^\eps_\sfT+1) \sup_{s \in [0,\sfT]} |W_s| \right\} \ind_{\Omega\setminus G^\eps_\sfN},\\
|\Theta^\eps_{\lambda^\eps_t}-\theta(t)| &\le \left\{2\delta+ 2\varepsilon^p \sup_{[0,\thinspace \sfT]} |B_t|\right\} \ind_{G^\eps_\sfN} + \left\{ 4\sfT + 2\eps^p \sup_{[0, \sfT]} |B_t|\right\} \ind_{\Omega\setminus G^\eps_\sfN}.
\end{aligned}
\end{equation*}
\end{proposition}

Proposition \ref{P:LLN-comparison} follows easily from the Lemmas below.

\begin{lemma}\label{L:LLN-comparison-G}
	For $\eps \in (0,1)$ small enough, $t \in [0,\sfT]$, we have
	\begin{equation*}
		\begin{aligned}
			|R^\eps_{\lambda^\eps_t}-r(t)|\cdot \ind_{G^\eps_\sfN} & \le C\left\{\|J^\eps_\cdot-1\|_\infty+\eps \sup_{0 \le s \le \sfT} |W_s|\right\}e^{Ct}\cdot \ind_{G^\eps_\sfN},\\
			|\Theta^\eps_{\lambda^\eps_t}-\theta(t)|\cdot \ind_{G^\eps_\sfN} &\le \left\{ \sum_{k \ge 1} \ind_{[t_{k-1},\thinspace t_k)}(t)\left(2\delta+ 2\varepsilon^p \sup_{[0,\thinspace \sfT]} |B_t|\right)\right\}\cdot \ind_{G^\eps_\sfN}.
		\end{aligned}
	\end{equation*}
	\quad where, $C$ depends on $\sfK_b$, $\|h^\prime\|$, $\sfT$, $\sfN$. 
\end{lemma}
\begin{lemma}\label{L:LLN-comparison-B}
	For $\eps \in (0,1)$ small enough, $t \in [0,\sfT]$, we have
	\begin{equation}\label{E:LLN-comparison-B}
		\begin{aligned}
			|R^\eps_{\lambda^\eps_t}-r(t)| \cdot \ind_{\Omega\setminus G^\eps_\sfN} &\le \left\{  e^{\lhp t/\alpha} (|r_0|+\sfK_b t) + e^{\lhp \Ren^\eps_t} \left(|r_0| + \sfK_b t + 2\eps(\Ren^\eps_t+1) \sup_{s \in [0,\sfT]}|W_s|\right)\right\} \ind_{\Omega \setminus G^\eps_\sfN},\\
			|\Theta^\eps_{\lambda^\eps_t}-\theta(t)| \cdot \ind_{\Omega\setminus G^\eps_\sfN} &\le \left\{ 4\sfT + 2\eps^p \sup_{[0, \sfT]} |B_t|\right\} \ind_{\Omega \setminus G^\eps_\sfN}.
		\end{aligned}
	\end{equation}
\end{lemma}
\begin{proof}[Proof of Proposition \ref{P:LLN-comparison}]
	The claim easily follows from Lemmas \ref{L:LLN-comparison-G} and \ref{L:LLN-comparison-B}.
\end{proof}
\begin{proof}[Proof of Theorem \ref{T:LLN}]
	For $\beta\in \{1, 2\}$, 
	\begin{equation*}
		\mathbb{E}\left[\sup_{0 \le t \le \sfT}|R^\eps_{\lambda^\eps_t}-r(t)|^\beta\right]= \mathbb{E}\left[\sup_{0 \le t \le \sfT}|R^\eps_{\lambda^\eps_t}-r(t)|^\beta \ind_{G^\eps_\sfN}\right]  +
		\mathbb{E}\left[\sup_{0 \le t\le\sfT}|R^\eps_{\lambda^\eps_t}-r(t)|^\beta\ind_{\Omega\setminus G^\eps_\sfN} \right] 
	\end{equation*}
	
	Taking the expectations in Proposition \ref{P:LLN-comparison} 
	\begin{equation*}
		\begin{aligned}
			\mathbb{E}\left[\sup_{0 \le t \le \sfT}|R^\eps_{\lambda^\eps_t}-r(t)|^\beta \ind_{G^\eps_\sfN}\right]\le C\,  \mathbb{E}\left[\|J^\eps_\cdot-1\|^\beta_\infty+\eps^\beta \sup_{0 \le s \le \sfT} |W_s|^\beta \right]\le C(\delta^\beta + \eps^\beta )
		\end{aligned}
	\end{equation*}
	where the last inequality is derived using Lemma \ref{L:Jacobian-Estimates} and Burkholder-Davis-Gundy inequality \cite{Oksendal}
	
	Again going back to Proposition \ref{P:LLN-comparison} and taking expectation yields, 
	
	\begin{multline*}
		\mathbb{E}\left[\sup_{0 \le t\le\sfT}|R^\eps_{\lambda^\eps_t}-r(t)|^\beta\ind_{\Omega\setminus G^\eps_\sfN} \right] \le \mathbb{E}\left[  Ce^{\varrho \Ren^\eps_\sfT}\left\{1 + \eps (\Ren^\eps_\sfT+1) \sup_{s \in [0,\sfT]} |W_s| \right\} ^\beta\ind_{\Omega\setminus G^\eps_\sfN}\right]\\
		\le\mathbb{E}\left[ \left\{ Ce^{\varrho \Ren^\eps_\sfT}\left\{1 + \eps (\Ren^\eps_\sfT+1) \sup_{s \in [0,\sfT]} |W_s| \right\} \right\}^{2\beta}\right]^{1/2}\mathbb{E}\left[ \ind_{\Omega \setminus G^\eps_\sfN}\right]^{1/2}
	\end{multline*}
	
	Now a few observations will helps to see that for some constant $C_\beta>0$ (indpendent of $\eps$) such that
	\begin{equation}\label{E:LLN-Constant-B}
		\mathbb{E}\left[ \left\{ Ce^{\varrho \Ren^\eps_\sfT}\left\{1 + \eps (\Ren^\eps_\sfT+1) \sup_{s \in [0,\sfT]} |W_s| \right\} \right\}^{2\beta}\right] \le C_\beta
	\end{equation} 
	
	Firstly we use $(a+b)^{2\beta}\le 2^{2\beta-1}(a^{2\beta}+b^{2\beta})$ to see
	\begin{equation*}
		\left\{ Ce^{\varrho \Ren^\eps_\sfT}\left\{1 + \eps (\Ren^\eps_\sfT+1) \sup_{s \in [0,\sfT]} |W_s| \right\} \right\}^{2\beta} \le C_\beta e^{2\beta(\varrho+1) \Ren^\eps_\sfT}\left(1+ \eps^{2\beta}  \sup_{s \in [0,\sfT]} |W_s|^{2\beta}\right)
	\end{equation*}
	
	Now the task reduces to show $$ \mathbb{E}\left[ e^{2\beta(\varrho+1) \Ren^\eps_\sfT}\left(1+ \eps^{2\beta}  \sup_{s \in [0,\sfT]} |W_s|^{2\beta}\right)\right] <\infty$$
	
	Use H\"older's inequality to get
	\begin{equation*}
		\mathbb{E}\left[ e^{2\beta(\varrho+1) \Ren^\eps_\sfT}\left(1+ \eps^{2\beta}  \sup_{s \in [0,\sfT]} |W_s|^{2\beta}\right)\right]^2\le \mathbb{E}\left[e^{4\beta(\varrho+1) \Ren^\eps_\sfT}\right] \mathbb{E}\left[\left(1+ \eps^{2\beta}  \sup_{s \in [0,\sfT]} |W_s|^{2\beta}\right)^2\right]
	\end{equation*}
	
	A simple algebra and the use of Burkholder-Davis-Gundy inequality \cite{Oksendal} will show that \begin{equation*}
		\mathbb{E}\left[\left(1+ \eps^{2\beta}  \sup_{s \in [0,\sfT]} |W_s|^{2\beta}\right)^2\right]\le C_\beta \sfT_\beta
	\end{equation*}
	
	Recalling Lemma \ref{L:Renewal-Process}
	\begin{equation*}
		\begin{aligned}
			\mathbb{E}\left[e^{4\beta(\varrho+1) \Ren^\eps_\sfT}\right]\le \kappa(\lambda) e^{\lambda \sfT},   \text{ for some constant $\kappa(\lambda)$, where $\lambda> \frac{4\beta(\varrho+1)}{\alpha}+ \frac{(4\beta(\varrho+1))^2}{2\alpha^2} $} 
		\end{aligned}
	\end{equation*}
	
	Stitching everything together we are able to show \eqref{E:LLN-Constant-B}.
	
	Therefore,  
	\begin{equation*}
		\mathbb{E}\left[\sup_{0 \le t\le\sfT}|R^\eps_{\lambda^\eps_t}-r(t)|^\beta\ind_{\Omega\setminus G^\eps_\sfN} \right] 
		\le C_\beta\,  \mathbb{P}\left(\Omega \setminus G^\eps_\sfN\right)^{1/2} , \text{for some constant $C_\beta >0$}
	\end{equation*}
It is to be noted that \begin{equation}\label{E:small-asymp}
	\exp\left(-\frac{1}{8\alpha}\frac{\delta^2}{\eps^{2p}}\right)
	= o(\eps^N),
\end{equation}
for any fixed $p>1$, any $\nu \in (1,p)$, and any positive integer $N$, provided $\eps \in (0,1)$ is sufficiently small, where $\delta = \eps^\nu$.

	Consequently, 
	\begin{equation*}
		\begin{aligned}
			\mathbb{E}\left[\sup_{0 \le t \le \sfT}|R^\eps_{\lambda^\eps_t}-r(t)|^\beta \right]& \le C_\beta\left(\delta^\beta + \eps^\beta +  \mathbb{P}\left(\Omega \setminus G^\eps_\sfN\right)^{1/2} \right) \\
			& \le C_\beta\left(\delta^\beta + \eps^\beta + \sqrt{\frac{\eps^p}{\delta}} \exp\left(-\frac{1}{8\alpha}\frac{\delta^2}{\eps^{2p}}\right)\right)\le C_\beta \eps^\beta 
		\end{aligned}
	\end{equation*}

We next delve ourselves in proving \begin{equation*}
		\mathbb{E}\left[\sup_{0 \le t \le \sfT}|\Theta^\eps_{\lambda^\eps_t}-\theta(t)|^\beta \right]\le C_\beta \eps^\beta
	\end{equation*}
Visiting back to the Proposition \ref{P:LLN-comparison} and taking expectation
	\begin{equation*}
		\begin{aligned}
			&\mathbb{E}\left[\sup_{0 \le t \le \sfT}|\Theta^\eps_{\lambda^\eps_t}-\theta(t)|^\beta\right]\\
			&\le \mathbb{E}\left[\sup_{0 \le t \le \sfT}|\Theta^\eps_{\lambda^\eps_t}-\theta(t)|^\beta \ind_{G^\eps_\sfN}\right]  +
			\mathbb{E}\left[\sup_{0 \le t\le\sfT}|\Theta^\eps_{\lambda^\eps_t}-\theta(t)|^\beta\ind_{\Omega\setminus G^\eps_\sfN} \right] \\
			&\le C_\beta\mathbb{E}\left[\left(\delta^\beta+ \varepsilon^{p\beta}\sup_{[0,\thinspace \sfT]} |B_t|^\beta\right) \ind_{G^\eps_\sfN}+ \left((4\sfT)^\beta + (2\eps^p)^\beta \sup_{[0, \sfT]} |B_t|^\beta\right)\ind_{\Omega\setminus G^\eps_\sfN}\right]\\
			&\le C_\beta \left(\delta^\beta +\eps^{p\beta} \mathbb{E}\left(\sup_{[0,\thinspace \sfT]} |B_t|^\beta\right)\right)+\left[ \mathbb{E}\left((4\sfT)^\beta + 2\eps^{p\beta} \sup_{[0, \sfT]} |B_t|^\beta\right)^2\right]^{1/2} \left(\mathbb{P}(\Omega\setminus G^\eps_\sfN)\right)^{1/2}
		\end{aligned}
	\end{equation*}
	\noindent
	Again the use of Burkholder-Davis-Gundy \cite{KS91} inequality will provide $\mathbb{E}\left[\sup_{[0,\thinspace \sfT]} |B_t|^\beta\right]\le C_\beta \sfT^{\beta/2} $, for some constant $C_\beta>0$.\\
	\noindent
	Opening the square and using Burkholder-Davis-Gundy inequality \cite{Oksendal} will show us 
	
	\begin{equation*}
		\begin{aligned}
			\mathbb{E}\left((4\sfT)^\beta + (2\eps)^{p\beta} \sup_{[0, \sfT]} |B_t|^\beta\right)^2\le C_\beta^2,
		\end{aligned}
	\end{equation*}
	\text{for some constant $C_\beta>0$}. 
	Therefore, 
	\begin{equation*}
		\begin{aligned}
			\mathbb{E}\left[\sup_{0 \le t \le \sfT}|\Theta^\eps_{\lambda^\eps_t}-\theta(t)|^\beta\right]\le C\left(\delta^\beta +\eps^{p\beta}+ \sqrt{\frac{\eps^p}{\delta}} \exp\left(-\frac{1}{8\alpha}\frac{\delta^2}{\eps^{2p}}\right)\right)\le C\eps^\beta \quad \text{, (using \ref{E:small-asymp})}
		\end{aligned}
	\end{equation*}
	
	Consequently,
	\begin{equation*}
		\mathbb{E}\left[\sup_{0 \le t\le\sfT} |X^\eps_t- x(t) |\right]\le  C\eps^\beta
	\end{equation*}
\end{proof}


\section{Analysis of Fluctuations}\label{S:CLT}
Analogous to Section \ref{S:LLN}, the current contains the proof of Theorem \ref{T:CLT}, splitting the reasoning into several lemmas.We begin by stating several lemmas that facilitate the estimation of the quantity. $|R^\eps_{\lambda^\eps_t}-r(t)-\eps R^1_t|$ on $G^\eps_\sfN$ and $\Omega\setminus G^\eps_\sfN$. Again the proofs of the lemmas are given in Section \ref{S:Lemmas}
\begin{proposition}\label{P:CLT-comparison}
There exists a constant $C=C(\sfT)>0$ such that for $\eps \in (0,1)$, we have
\begin{equation*}
\begin{aligned}
|R^\eps_{\lambda^\eps_t}-r(t)-\eps R^1_t| &\le C \left\{\|J^\eps_\cdot-1\|_\infty+ \sup_{[0, \sfT]} |R^\eps_{\lambda^\eps_t} - r(t)|^2 +\eps \sup_{\substack{0 \le s<t \le \sfT\\ |t-s|<\delta}} |W_t - W_s|\right\} \ind_{G^\eps_\sfN} \\ & \qquad+Ce^{\varrho \Ren^\eps_t}\left\{1 + (\eps (\Ren^\eps_\sfT+1)+1) \sup_{s \in [0,\sfT]} |W_s| \right\} \ind_{\Omega\setminus G^\eps_\sfN}
\end{aligned}
\end{equation*}
\end{proposition}

Proposition \ref{P:CLT-comparison} follows easily from the Lemmas below.

\begin{lemma}\label{L:CLT-comparison-G}
There exists a constant $C_{\ref{L:CLT-comparison-G}}>0$ depending on $\sfK_b$, $\|h^\prime\|$, $\|h^{\prime\prime}\|$,  $\sfT$, $\sfN$ such that for $\eps \in (0,1)$ small enough, $t \in [0,\sfT]$, we have
\begin{equation*}
\begin{aligned}
|R^\eps_{\lambda^\eps_t}-r(t)-\eps R^1_t|\cdot \ind_{G^\eps_\sfN} & \le C_{\ref{L:CLT-comparison-G}}\left\{\|J^\eps_\cdot-1\|_\infty+ \sup_{[0, \sfT]} |R^\eps_{\lambda^\eps_t} - r(t)|^2 + \int_0^t \sup_{0 \le u \le s}|R^\eps_{\lambda^\eps_u}-r(u)|^2 \thinspace ds \right.\\& \left. +\eps \sup_{\substack{0 \le s<t \le \sfT\\ |t-s|<\delta}} |W_t - W_s|\right\}e^{C_{\ref{L:CLT-comparison-G}}t}\cdot \ind_{G^\eps_\sfN}
\end{aligned}
\end{equation*}
\end{lemma}
\begin{remark}\label{r:angular}
	It is to be noted that the stated lemmas and proposition does not involve any estimates in the angular coordinate. The reason for that is that the fluctuation process $\Theta^1\equiv 0$, for any $p>1$, leading to the same estimates as stated in the lemmas \ref{L:LLN-comparison-G} and \ref{L:LLN-comparison-B}, and consequently to estimates in the proof of Theorem \ref{T:LLN} 
\end{remark}

\begin{lemma}\label{L:CLT-comparison-B}
For $\eps \in (0,1)$ small enough, $t \in [0,\sfT]$, we have
\begin{equation*}
\begin{aligned}
|R^\eps_{\lambda^\eps_t}-r(t)-\eps R^1_t| \cdot \ind_{\Omega\setminus G^\eps_\sfN} & \le   \left\{  e^{\lhp t/\alpha} (|r_0|+\sfK_b t) + e^{\lhp \Ren^\eps_t} \left(|r_0| + \sfK_b t + 2\eps(\Ren^\eps_t+1) \sup_{t \in [0,\sfT]}|W_t|\right) \right. \\ & \left. \qquad + \eps\,  2(\ren(t)+1) e^{(2\lhp/\alpha + \sfK_b e^{\lhp t/\alpha})t}  \sup_{[0, \sfT]} |W_t|\right\} \ind_{\Omega \setminus G^\eps_\sfN}
\end{aligned}
\end{equation*}
\end{lemma}
 
 \begin{proof}[Proof of Proposition \ref{P:CLT-comparison}]
 	The claim easily follows from Lemmas \ref{L:CLT-comparison-G} and \ref{L:CLT-comparison-B}.
 \end{proof}
 To proceed, we will find it useful to recall some notation and results from \cite{baldi1992large} which will help us estimate moments of the H\"older norm of Brownian motion. 
 For $\alpha \in (0,1]$, let $\mathcal{C}^\alpha$ denote the Banach space of all $\alpha$-H\"older paths 
 $x : [0,\sfT] \to \mathbb{R}$ such that $x(0)=0$, equipped with the norm
 \[
 \|x\|_\alpha 
 = \sup_{0 \le s<t \le \sfT} \frac{|x(t)-x(s)|}{|t-s|^\alpha}, \quad \text{and for $\delta \in (0,1)$, set} \quad \sfw_x(\delta)
 = \sup_{\substack{0 \le s< t \le \sfT \\ |t-s|\le \delta}}
 \frac{|x(t)-x(s)|}{|t-s|^\alpha};
 \]
 the modulus of continuity of $x$ is thus $\delta^\alpha \sfw_x(\delta)$.
 Let $\mathcal{C}^{\alpha,0}$ denote the closed subspace of $\mathcal{C}^\alpha$ consisting of all paths $x$ such that $\displaystyle
 \lim_{\delta \to 0} \sfw_x(\delta)=0$; we note that the Banach space $\mathcal{C}^{\alpha,0}$ is separable. 
 Owing to the fact that on every bounded interval $[0,\sfT]$, the sample paths of Brownian motion $B_t$ are Holder-continuous with exponent $\alpha$ for every $\alpha < 1/2$, the paths of $B_t$ belong to $\mathcal{C}^{\alpha,0}$. We would like to show $\BE[\|B\|_\alpha]<\infty$. To this end, Fernique's theorem, stated next, will be useful.
 
 \begin{theorem}[Fernique's  theorem {\cite{fernique1970integrabilite}}]
 	Let $(E,\|\cdot\|)$ be a \emph{separable Banach space} and let $\mu$ be a \emph{centered Gaussian measure} on $E$.
 	Equivalently, let $X$ be an $E$-valued random variable such that for every bounded linear functional $\ell \in E^{*}$, the real random variable $\ell(X)$ is centered Gaussian. There exists a constant $v>0$ such that
 	$\displaystyle
 	\int_E \exp\!\bigl(v \|x\|^2\bigr)\, d\mu(x) < \infty,
 	$ or equivalently, $ \mathbb{E}\!\left[\exp\!\bigl(v \|X\|^2\bigr)\right] < \infty.$ In particular, $\|X\|$ has finite moments of all orders:
 	$
 	\mathbb{E}[\|X\|^m] < \infty \text{ for all } m>0.
 	$
 \end{theorem}

 \begin{lemma}\label{L:Holder-norm}
 	Let $B_t$ be a standard one-dimensional Brownian motion. Then, for any fixed $\sfT>0$, $\alpha \in (0,\tfrac{1}{2})$, we have
 	\begin{equation}\label{E:Holder-norm-moment}
 		C_\alpha \triangleq \BE[\|B\|_\alpha] < \infty, \quad \text{and further} \quad 
 		\BE\left[\sup_{\substack{0 \le s < t \le \sfT\\|t-s| \le \delta}} |B_t - B_s| \right] \le C_\alpha \delta^\alpha \quad \text{for all $\delta \in (0,1)$.}
 	\end{equation}
 \end{lemma}
 
 \begin{proof}
 	Fix $\sfT>0$, $\alpha \in (0,\tfrac{1}{2})$. As remarked earlier, the paths of $B_t$ over $t \in [0,\sfT]$ can be viewed as random variables taking values in the separable Banach space $(\mathcal{C}^{\alpha, 0}, \|\cdot\|_\alpha)$.
 	Since $(\mathcal{C}^{\alpha, 0}, \|\cdot\|_\alpha)$ is compactly embedded in the space $(C[0,\sfT], \|\cdot\|_\infty)$ which already has a Gaussian measure, namely Wiener measure $\mathcal{W}$, defined on it. 
 	One can now see, using \cite{baldi2021intermediate} and Fernique's theorem, that there exists $v>0$ such that $\mathbb{E}\big[\exp(v\|B\|^2_\alpha)\big] < \infty$, implying in particular finiteness of all moments of $\|B\|_\alpha$. This proves the first assertion in \eqref{E:Holder-norm-moment}. To prove the second part of \eqref{E:Holder-norm-moment}, we note that 
 	for any $\delta \in (0,1)$, we have $\displaystyle \sup_{\substack{0 \le s < t \le \sfT\\|t-s| \le \delta}} |B_t - B_s| \le \sup_{\substack{0 \le s < t \le \sfT\\|t-s| \le \delta}} \frac{|B_t - B_s|}{|t-s|^\alpha} \delta^\alpha = \|B\|_\alpha \delta^\alpha$. Taking expectations, we get the stated result.
 \end{proof}

 \begin{proof}[Proof of Theorem \ref{T:CLT}]
 	Taking the expectations in Proposition \ref{P:CLT-comparison}
 	\begin{equation*}
 		\mathbb{E}\left[\sup_{0 \le t \le \sfT}|R^\eps_{\lambda^\eps_t}-r(t)-\eps  R^1_t|\right]= \mathbb{E}\left[\sup_{0 \le t \le \sfT}|R^\eps_{\lambda^\eps_t}-r(t)-\eps R^1_t| \ind_{G^\eps_\sfN}\right]  +
 		\mathbb{E}\left[\sup_{0 \le t\le\sfT}|R^\eps_{\lambda^\eps_t}-r(t)-\eps R^1_t|\ind_{\Omega\setminus G^\eps_\sfN} \right] 
 	\end{equation*}
 	
 	\begin{multline*}
 		\mathbb{E}\left[\sup_{0 \le t \le \sfT}|R^\eps_{\lambda^\eps_t}-r(t)-\eps R^1_t| \ind_{G^\eps_\sfN}\right]\le C\,  \mathbb{E}\left[\|J^\eps_\cdot-1\|_\infty+ \sup_{0 \le t \le \sfT} |R^\eps_{\lambda^\eps_t}-r(t)|^2 +\eps \sup_{\substack{0\le s<t\le\sfT \\ |t-s|\le \delta}} |W_t-W_s| \right]\\
 		\le C(\delta +\eps^2 +\eps \delta^\alpha)
 	\end{multline*}
 	where the last inequality is derived from using Lemma \ref{L:Jacobian-Estimates}, Theorem \ref{T:LLN}, and Lemma \ref{L:Holder-norm}
 	
 	The next goal is show that there exits a constant $C>0$ (independent of $\eps$) such that \begin{equation}\label{E:CLT-R-Constant-B}
 		\mathbb{E}\left[\sup_{0 \le t\le\sfT}|R^\eps_{\lambda^\eps_t}-r(t)-\eps R^1_t| \ind_{\Omega\setminus G^\eps_\sfN} \right] < C
 	\end{equation}
 	
 	Again, the job is almost done as $$\mathbb{E}\left[\sup_{0 \le t\le\sfT}|R^\eps_{\lambda^\eps_t}-r(t)-\eps R^1_t|\ind_{\Omega\setminus G^\eps_\sfN} \right] \le 	\mathbb{E}\left[\sup_{0 \le t\le\sfT}|R^\eps_{\lambda^\eps_t}-r(t)|\ind_{\Omega\setminus G^\eps_\sfN} \right] + \eps	\mathbb{E}\left[\sup_{0 \le t\le\sfT}| R^1_t|\ind_{\Omega\setminus G^\eps_\sfN} \right] $$ and we have already established in Theorem \ref{T:LLN},  $\mathbb{E}\left[\sup_{0 \le t\le\sfT}|R^\eps_{\lambda^\eps_t}-r(t)|\ind_{\Omega\setminus G^\eps_\sfN} \right]<\infty$
 	
 	So we will only focus on proving $	\mathbb{E}\left[\sup_{0 \le t\le\sfT}| R^1_t|\ind_{\Omega\setminus G^\eps_\sfN} \right]<\infty$
 	
 	Revisiting the following expression at the end of proof in Lemma \ref{L:CLT-comparison-B}

 	\begin{equation*}
 		\mathbb{E}\left[\sup_{0 \le t \le \sfT}|R^1_t|^2\right]\le  2(\ren(\sfT)+1)^2 e^{2(2\lhp/\alpha + \sfK_b e^{\lhp \sfT/\alpha})\sfT}  \mathbb{E}\left[\sup_{[0, \sfT]} |W_t|^2\right]\le C
 	\end{equation*} 
 	where $C>0$ is some constant obtained from the definiteness of $2(\ren(\sfT)+1)^2 e^{2(\lhp/\alpha + \sfK_b e^{\lhp \sfT/\alpha})\sfT}$ and using Burkholder-Davis-Gundy inequality $\mathbb{E}\left[\sup_{[0, \sfT]} |W_t|^2\right]<\infty$, establishing \eqref{E:CLT-R-Constant-B}.

 	Putting Everything together, returns,
 	
 	\begin{equation*}
 		\mathbb{E}\left[\sup_{0 \le t \le \sfT}|R^\eps_{\lambda^\eps_t}-r(t)-\eps  R^1_t|\right]\le C\left(\delta+\eps^2+\eps\delta^{\upsilon}+ \sqrt{\frac{\eps^p}{\delta}} \exp\left(-\frac{1}{8\alpha}\frac{\delta^2}{\eps^{2p}}\right)\right)
 	\end{equation*}
Recalling Remark \ref{r:angular} together with the estimates just obtained yields
 	\begin{equation*}
 		\mathbb{E}\left[\sup_{0 \le t\le\sfT} |X^\eps_t- x(t) -\eps Z_t|\right]\le C \left(\delta+\eps^{2}+\eps\delta^\upsilon + \sqrt{\frac{\eps^p}{\delta}} \exp\left(-\frac{1}{8\alpha}\frac{\delta^2}{\eps^{2p}}\right)\right)
 	\end{equation*}

 	We take $\delta=o(\eps^\nu)$ where $p>\nu>1$ which gets us,

 	\begin{equation*}
 		\mathbb{E}\left[\sup_{0 \le t\le\sfT} |X^\eps_t- x(t) -\eps Z_t|\right]\le C \left(\eps^\nu+\eps^{2}+\eps^{1+\nu\upsilon} + \sqrt{\eps^{p-\nu}} \exp\left(-\frac{1}{8\alpha}\frac{1}{\eps^{2(p-\nu)}}\right)\right)
 	\end{equation*}
With the reference of \eqref{E:small-asymp}, when $p>2$, we take $\nu>2$ and $\upsilon=1/\nu$ to obtain
 	
 	\begin{equation*}
 		\mathbb{E}\left[\sup_{0 \le t\le\sfT} |X^\eps_t- x(t) -\eps Z_t|\right]\le C\eps^2,
 	\end{equation*}
 	
 and when $p\le 2$, we take $\upsilon= 1- 1/\nu $ to obtain
 	
 	\begin{equation*}
 		\mathbb{E}\left[\sup_{0 \le t\le\sfT} |X^\eps_t- x(t) -\eps Z_t|\right]\le C\eps^\nu
 	\end{equation*}
 	
 \end{proof}
 

\section{Proofs of Lemmas}\label{S:Lemmas}

\begin{proof}[Proof of Lemma \ref{L:LLN-comparison-G}]
	If $\omega \in G^\eps_\sfN$, $t \in [0,\sfT]$, then it is easily seen from \eqref{E:stoch-state-int-eq} and \eqref{E:rtd-generic} that
	\begin{multline}\label{E:R-time-changed}
		R^\eps_{\lambda^\eps_t(\omega)}= \sum_{k=1}^{\sfN} \ind_{[t_{k-1},t_k)}(t)\left\{R^{\eps,+}_{\tau^\eps_{k-1}(\omega)} + \int_{t_{k-1}}^t b(R^\eps_{\lambda^\eps_u(\omega)})J^\eps_{u}(\omega) \thinspace du + \eps\left(W_{\lambda^\eps_t(\omega)}-W_{\tau_{k-1}^\eps(\omega)}\right)\right\}\\
		+ \ind_{[t_{\sfN},\sfT]}(t)\left\{R^{\eps,+}_{\tau^\eps_{\sfN}(\omega)} + \int_{t_{\sfN}}^t b(R^\eps_{\lambda^\eps_u(\omega)})J^\eps_{u}(\omega) \thinspace du + \eps\left(W_{\lambda^\eps_t(\omega)}-W_{\tau_{\sfN}^\eps(\omega)}\right)\right\},
	\end{multline}
	where we have made use of a simple change of variables to obtain $\int_{\tau^\varepsilon_{k-1}}^{\lambda^\varepsilon_t} b(R^{\varepsilon}_s) \ ds=\int_{t_{k-1}}^{t}  b(R^{\varepsilon}_{\lambda^\varepsilon_s}) J_{\lambda^\varepsilon_s} \ ds$.	
	Now suppose $t \in [t_{k-1}, t_{k})$ for some $1 \le k \le \sfN$ or $t \in [t_{k-1}, \sfT]$ with $k=\sfN+1$. Then, we see from equations \eqref{E:det-state-int-eq} and \eqref{E:R-time-changed}  that $R^{\varepsilon}_{\lambda^\varepsilon_t} - r(t)
	= R^{\varepsilon,+}_{\tau^\varepsilon_{k-1}} - r^+(t_{k-1}) + \int_{t_{k-1}}^{t}  b(R^{\varepsilon}_{\lambda^\varepsilon_s}) J_{\lambda^\varepsilon_s} \ ds - \int_{t_{k-1}}^{t} b(r(s)) \ ds + \varepsilon (W_{\lambda^\varepsilon_t}-W_{\tau^\varepsilon_{k-1}})$. Adding and subtracting $\int_{t_{k-1}}^{t}  b(R^{\varepsilon}_{\lambda^\varepsilon_s}) \ ds$, and using Taylor's formula, we get			
	\begin{multline}\label{E:R-error-k}
		R^{\varepsilon}_{\lambda^\varepsilon_t} - r(t) = \left\{R^{\varepsilon,+}_{\tau^\varepsilon_{k-1}} - r^+(t_{k-1})\right\} + \II^\eps_k(t), \qquad \text{where} \\
		\II^\eps_k(t) \triangleq \int_{t_{k-1}}^{t} b(R^{\varepsilon}_{\lambda^\varepsilon_s}) (J_{\lambda^\varepsilon_s}-1) \ ds + \int_{t_{k-1}}^{t} b^{\prime}(\xi^{k,\varepsilon}_s)\left(R^{\varepsilon}_{\lambda^\varepsilon_s}-r(s)\right) \ ds + \varepsilon (W_{\lambda^\varepsilon_t}-W_{\tau^\varepsilon_{k-1}})
	\end{multline}
	where $\xi^{k,\varepsilon}_s$ is a point between $r(s)$ and $R^\varepsilon_{\lambda^\varepsilon_s}$. As $t \nearrow t_k$, $1 \le k \le \sfN$, the expression above approaches $R^{\eps,-}_{\tau^\eps_k}-r^-(t_k)$. Recalling the resetting rule in the last lines of \eqref{E:det-state-int-eq} and \eqref{E:stoch-state-int-eq}, we use Taylor's formula to get 
	\begin{equation}\label{E:R-resetting-k} 
		R^{\varepsilon,+}_{\tau^\varepsilon_{k}} - r^+(t_{k})=h(R^{\varepsilon,-}_{\tau^\varepsilon_{k}})-h(r^-(t_{k}))=h^\prime(\eta^{k,\varepsilon})(R^{\varepsilon,-}_{\tau^\varepsilon_{k}}-r^-(t_{k})), 
	\end{equation} 
	where $\eta^{k,\varepsilon}$ is a point between $r^-(t_{k})$ and $R^{\varepsilon,-}_{\tau^\varepsilon_{k}}$.  Starting from the interval $[t_0,t_1)$ and working forward in time, successive alternate application of \eqref{E:R-error-k} followed by \eqref{E:R-resetting-k} yields that for $t \in [t_{k-1},t_k)$ with $1 \le k \le \sfN$, or $t \in [t_{k-1},\sfT]$ with $k=\sfN+1$, we have
	\begin{equation*}
		R^\eps_{\lambda^\eps_t}-r(t)= \prod_{j=1}^{k-1}h^\prime(\eta^{j,\eps})\left\{R^{\eps,+}_{\tau^\eps_0}-r^+(t_0)\right\} + \sum_{i=1}^{k-1} \left(\prod_{j=i}^{k-1} h^\prime(\eta^{j,\eps})\right) \II^\eps_i(t_i) + \II^\eps_k(t),
	\end{equation*}
	where we use the conventions that $\sum_{i=1}^0 (\dots) \triangleq 0$ and $\prod_{j=1}^0 (\dots) \triangleq 1$. Unwrapping this expression and suitably regrouping the terms, we see that for $t \in [t_{k-1},t_k)$ with $1 \le k \le \sfN$ or $t \in [t_{k-1},\sfT]$ with $k=\sfN+1$, we have
	\begin{equation}\label{E:R-error-G}
		\begin{aligned}
			R^\eps_{\lambda^\eps_t}-r(t) &= \prod_{j=1}^{k-1}h^\prime(\eta^{j,\eps})\left\{R^{\eps,+}_{\tau^\eps_0}-r^+(t_0)\right\} + \JJ^\eps_k(t) + \LL^\eps_k(t) + \MM^\eps_k(t), \quad \text{where} \\  
			\JJ^\eps_k(t) &\triangleq \sum_{i=1}^{k-1} \left(\prod_{j=i}^{k-1} h^\prime(\eta^{j,\eps})\right) \int_{t_{i-1}}^{t_i} b(R^\eps_{\lambda^\eps_s}) (J^\eps_{s}-1) \thinspace ds
			+ \int_{t_{k-1}}^t b(R^\eps_{\lambda^\eps_s})\left(J^\eps_{s} - 1 \right) \thinspace ds \\ 
			\LL^\eps_k(t) &\triangleq  \sum_{i=1}^{k-1} \left(\prod_{j=i}^{k-1} h^\prime(\eta^{j,\eps})\right) \int_{t_{i-1}}^{t_i} b^\prime(\xi^{i,\eps}_s) \left(R^\eps_{\lambda^\eps_s}-r(s)\right) \thinspace ds + \int_{t_{k-1}}^t b^\prime(\xi^{k,\eps}_s)\left(R^\eps_{\lambda^\eps_s}-r(s)\right) \thinspace ds\\
			\MM^\eps_k(t) &\triangleq  \eps \sum_{i=1}^{k-1} \left(\prod_{j=i}^{k-1} h^\prime(\eta^{j,\eps})\right) \left(W_{\tau^\eps_i}-W_{\tau^\eps_{i-1}}\right)
			+ \eps\left(W_{\lambda^\eps_t}-W_{\tau^\eps_{k-1}}\right).
		\end{aligned}
	\end{equation}
	Recalling Assumptions \ref{A:vf-wedge} and \ref{A:resetting-map}, we now carefully estimate the terms on the right hand side of the first line in \eqref{E:R-error-G}. The first term is zero on account of the initial conditions. 
	Next, we note that $|\JJ^\eps_k(t)| \le (\|h^\prime\|^{k-1}\vee 1)\sfK_b \int_0^t |J^\eps_{s}-1| \thinspace ds \le \sfK_b (\|h^\prime\|^{\sfN}\vee 1) \int_0^\sfT |J^\eps_{s}-1| \thinspace ds$.
	It is easy to see that $|\LL^\eps_k(t)| \le \sfK_b (\|h^\prime\|^{\sfN}\vee 1) \int_0^t \sup_{0 \le u \le s}|R^\eps_{\lambda^\eps_u}-r(u)| \thinspace ds$. Finally, $|\MM^\eps_k(t)| \le 2k \eps (\|h^\prime\|^{\sfN}\vee 1) \sup_{0 \le s \le \sfT} |W_s|$. Thus, there exists a constant $C>0$ (depending on $\sfK_b$, $\|h^\prime\|$, $\sfT$, $\sfN$) such that $|R^\eps_{\lambda^\eps_t}-r(t)| \le C\left\{\|J^\eps_\cdot-1\|_\infty + \int_0^t \sup_{0 \le u \le s}|R^\eps_{\lambda^\eps_u}-r(u)| \thinspace ds + \eps \sup_{0 \le s \le \sfT} |W_s|\right\}$. Since the right-hand side of the latter is non-decreasing in $t$, we have
	\begin{equation*}
		\sup_{0 \le s \le t} |R^\eps_{\lambda^\eps_s}-r(s)| \le C\left\{\|J^\eps_\cdot-1\|_\infty + \int_0^t \sup_{0 \le u \le s}|R^\eps_{\lambda^\eps_u}-r(u)| \thinspace ds + \eps \sup_{0 \le s \le \sfT} |W_s|\right\}.
	\end{equation*}
	
	Applying Gronwall's inequality to the above equation, yields,
	\begin{equation*}
		\sup_{0 \le s \le t} |R^\eps_{\lambda^\eps_s}-r(s)|\le C\left\{\|J^\eps_\cdot-1\|_\infty+\eps \sup_{0 \le s \le \sfT} |W_s|\right\}e^{Ct}
	\end{equation*}
	
	Next, we estimate the close proximity of angular velocities in the Good sets. 
	
	From equations \eqref{E:det-state-int-eq} and \eqref{E:stoch-state-int-eq}
	\begin{equation}
		\begin{aligned}
			\Theta^\eps_{\lambda^\varepsilon_t}-\theta(t) &= \sum_{k \ge 1} \ind_{[\tau^\eps_{k-1},\tau^\eps_k)}(\lambda^\varepsilon_t) \left\{\int_{\lambda^\varepsilon_t\wedge \tau^\eps_{k-1}}^{\lambda^\varepsilon_t} 1 \thinspace ds + \eps^p\left(B_{\lambda^\varepsilon_t} - B_{\lambda^\varepsilon_t\wedge \tau^\eps_{k-1}}\right)\right\}- \sum_{k \ge 1} \ind_{[t_{k-1},t_k)}(t) (t-t_{k-1})\\
			&=\sum_{k \ge 1} \ind_{[t_{k-1},\thinspace t_k)}(t)\left\{\left(\lambda^\varepsilon_t-\lambda^\varepsilon_t\wedge \tau^\eps_{k-1}-(t-t_{k-1})\right)+\eps^p \left(B_{\lambda^\varepsilon_t} - B_{\lambda^\varepsilon_t\wedge \tau^\eps_{k-1}}\right)\right\}\\
			& \text{We merged the two summand into one by observing } \lambda^\varepsilon_t\in[\tau^\varepsilon_{k-1},\thinspace \tau^\varepsilon_k) \text{ iff } t\in[t_{k-1}, \thinspace t_k)
		\end{aligned}
	\end{equation}
	
	Therefore, using the triangle inequality, we obtain the following,
	\begin{equation*}
		\begin{aligned}
			\left|\Theta^\eps_{\lambda^\varepsilon_t}-\theta(t)\right|&\le \sum_{k \ge 1} \ind_{[t_{k-1},\thinspace t_k)}(t)\left\{\left|\lambda^\varepsilon_t-\lambda^\varepsilon_t\wedge \tau^\eps_{k-1}-(t-t_{k-1})\right| + \eps^p \left|B_{\lambda^\varepsilon_t} - B_{\lambda^\varepsilon_t\wedge \tau^\eps_{k-1}}\right|\right\}\\
			&\le \sum_{k \ge 1} \ind_{[t_{k-1},\thinspace t_k)}(t)\left\{\max_{t\in[t_{k-1}, t_k)} \left|\lambda^\varepsilon_t-\lambda^\varepsilon_t\wedge \tau^\eps_{k-1}-(t-t_{k-1})\right|+ 2\eps^p \sup_{[0, \sfT]} |B_t|\right\}\\
			&\leq \sum_{k \ge 1} \ind_{[t_{k-1},\thinspace t_k)}(t)\left(2\delta+ 2\varepsilon^p\sup_{[0,\thinspace \sfT]} |B_t|\right), \qquad (\text{Using Lemma \ref{L:distorion}})
		\end{aligned}
	\end{equation*}
\end{proof}

\begin{proof}[Proof of Lemma \ref{L:LLN-comparison-B}]

	Since $\lambda^\eps_t(\omega) \equiv t$ for $\omega \notin G^\eps_\sfN$, we have $R^\eps_{\lambda^\eps_t}=R^\eps_t$ for $\omega \notin G^\eps_\sfN$. Thus, $|R^\eps_{\lambda^\eps_t}-r(t)|\ind_{\Omega\setminus G^\eps_\sfN} \le \left(|R^\eps_t|+|r(t)|\right) \ind_{\Omega\setminus G^\eps_\sfN}$. Let us now separately estimate $|r(t)|$ and $|R^\eps_t|$. 
	Using equation \eqref{E:MVT-h} from Assumption \ref{A:resetting-map}, an easy induction argument yields that for $t \in [t_{k-1},t_k)$, $k \in \BN$, we have $r(t)=\left(\prod_{j=1}^{k-1} h^\prime(\eta_j)\right)r^+(t_0) + \sum_{i=1}^{k-1} \left(\prod_{j=i}^{k-1} h^\prime(\eta_j)\right) \int_{t_{i-1}}^{t_i} b(r(s)) \thinspace ds + \int_{t_{k-1}}^t b(r(s)) \thinspace ds$ where $\eta_j$ is a number between $0$ and $r^-(t_j)$. Recalling equation \eqref{E:renRen}, we can now write
	\begin{equation}\label{E:r-B}
		r(t) = \left(\prod_{j=1}^{\ren(t)} h^\prime(\eta_j)\right)r^+(t_0) + \sum_{i=1}^{\ren(t)} \left(\prod_{j=i}^{\ren(t)} h^\prime(\eta_j)\right) \int_{t_{i-1}}^{t_i} b(r(s)) \thinspace ds + \int_{t_{\ren(t)}}^t b(r(s)) \thinspace ds 
	\end{equation}
	Similarly, we get that for $t \in [\tau^\eps_{k-1},\tau^\eps_k)$, $k \in \BN$, we have $R^\eps_t = \left(\prod_{j=1}^{k-1} h^\prime(\eta_j^\eps)\right)R^{\eps,+}_{\tau^\eps_0} \\+ \sum_{i=1}^{k-1} \left(\prod_{j=i}^{k-1} h^\prime(\eta_j^\eps)\right) \int_{\tau^\eps_{i-1}}^{\tau^\eps_i} b(R^\eps_s) \thinspace ds + \eps \sum_{i=1}^{k-1} \left(\prod_{j=i}^{k-1} h^\prime(\eta_j^\eps)\right) \left(W_{\tau^\eps_i}-W_{\tau^\eps_{i-1}}\right) + \int_{\tau^\eps_{k-1}}^t b(R^\eps_s) \thinspace ds \\+ \eps\left(W_t - W_{\tau^\eps_{k-1}}\right)$, where $\eta_j^\eps$ is a number between $0$ and $R^{\eps,-}_{\tau^\eps_j}$. Once again, by \eqref{E:renRen}, we write
	\begin{multline}\label{E:R-B}
		R^\eps_t =  \left(\prod_{j=1}^{\Ren^\eps_t} h^\prime(\eta_j^\eps)\right)R^{\eps,+}_{\tau^\eps_0} + \sum_{i=1}^{\Ren^\eps_t} \left(\prod_{j=i}^{\Ren^\eps_t} h^\prime(\eta_j^\eps)\right) \int_{\tau^\eps_{i-1}}^{\tau^\eps_i} b(R^\eps_s) \thinspace ds  + \int_{\tau^\eps_{\Ren^\eps_t}}^t b(R^\eps_s) \thinspace ds \\+ \eps \sum_{i=1}^{\Ren^\eps_t} \left(\prod_{j=i}^{\Ren^\eps_t} h^\prime(\eta_j^\eps)\right) \left(W_{\tau^\eps_i}-W_{\tau^\eps_{i-1}}\right) + \eps\left(W_t - W_{\tau^\eps_{\Ren^\eps_t}}\right).
	\end{multline}
	Recalling Assumption \ref{A:vf-wedge}, we note that $|R^\eps_t| \le \|h^\prime\|_\infty^{\Ren^\eps_t}|r_0| + \sfK_b\sum_{i=1}^{\Ren^\eps_t}\|h^\prime\|_\infty^{\Ren^\eps_t-(i-1)} (\tau^\eps_i-\tau^\eps_{i-1}) + \sfK_b(t-\tau^\eps_{\Ren^\eps_t}) + \eps \sum_{i=1}^{\Ren^\eps_t}\|h^\prime\|_\infty^{\Ren^\eps_t-(i-1)}|W_{\tau^\eps_i}-W_{\tau^\eps_{i-1}}| + \eps|W_t-W_{\tau^\eps_{\Ren^\eps_t}}|$. A simple calculation now yields
	\begin{multline}\label{E:R-B-estimate}
		|R^\eps_t| \le \left(\|h^\prime\|_\infty\vee 1\right)^{\Ren^\eps_t}\left\{|r_0| + \sfK_b t + \eps\sum_{i=1}^{\Ren^\eps_t}|W_{\tau^\eps_i}-W_{\tau^\eps_{i-1}}| + \eps|W_t-W_{\tau^\eps_{\Ren^\eps_t}}|\right\}\\
		\le \left(\|h^\prime\|_\infty\vee 1\right)^{\Ren^\eps_t}\left\{|r_0| + \sfK_b t + 2\eps(\Ren^\eps_t+1)\sup_{s \in [0,\sfT]}|W_s|\right\},
	\end{multline}
	where the latter follows from the triangle inequality. A similar simpler calculation yields
	\begin{equation}\label{E:r-B-estimate}
		|r(t)| \le \left(\|h^\prime\|_\infty\vee 1\right)^{\ren(t)}\left\{|r_0| + \sfK_b t\right\}.
	\end{equation}
	Recalling \eqref{E:log-h-prime}, equations \eqref{E:R-B-estimate} and \eqref{E:r-B-estimate} easily yield \eqref{E:LLN-comparison-B}.

	We first calculate the proximity of the angular velocities when {\sc bm} added to $\dot{\Theta}$
	\begin{equation*}
		\begin{aligned}
			\Theta^\eps_t-\theta(t) &= \sum_{k \ge 1} \ind_{[\tau^\eps_{k-1},\tau^\eps_k)}(t) \left\{\int_{t\wedge \tau^\eps_{k-1}}^{t} 1 \thinspace ds + \eps^p \left(B_{t} - B_{t\wedge \tau^\eps_{k-1}}\right)\right\}- \sum_{k \ge 1} \ind_{[t_{k-1},t_k)}(t) (t-t_{k-1})
		\end{aligned}
	\end{equation*}
	
	Therefore, using the triangle inequality, we obtain the following,
	\begin{equation}
		\begin{aligned}
			\left|\Theta^\varepsilon_t- \theta(t)\right|&\le  \sum_{k \ge 1} \ind_{[\tau^\eps_{k-1},\tau^\eps_k)}(t) \left(|t - t\wedge\tau^\varepsilon_{k-1}| +\eps^p |B_t - B_{t\wedge \tau^\varepsilon_{k-1}}|\right)\\
			& + \sum_{k\ge 1}\ind_{[t_{k-1}, t_k)} \left(|t- t_{k-1}|\right)\\
			&\le \left(4\sfT + 2\eps^p \sup_{[0, \sfT]} |B_t|\right)
		\end{aligned}
	\end{equation}
\end{proof}
\begin{proof}[Proof of Lemma \ref{L:CLT-comparison-G}]
	Let $\omega \in G^\eps_\sfN$, and suppose that $t \in [t_{k-1}, t_{k})$ for some $1 \le k \le \sfN$ or $t \in [t_{k}, \sfT]$ with $k=\sfN+1$. Then, we see from equations \eqref{E:det-state-int-eq}, \eqref{E:fluct-proc} and \eqref{E:R-time-changed}  that $R^{\varepsilon}_{\lambda^\varepsilon_t} - r(t)- \eps R^1_t
	= R^{\varepsilon,+}_{\tau^\varepsilon_{k-1}} - r^+(t_{k-1})-\eps R^{1,+}_{t_{k-1}} + \int_{t_{k-1}}^{t}  b(R^{\varepsilon}_{\lambda^\varepsilon_s}) J_{s} \ ds - \int_{t_{k-1}}^{t} \left[b(r(s))+\eps b^\prime(r(s))R^1_s\right]\ ds + \varepsilon (W_{\lambda^\varepsilon_t}-W_{\tau^\varepsilon_{k-1}}) - \eps (W_t - W_{t_{k-1}})$. 
	Adding and subtracting $\int_{t_{k-1}}^{t}  b(R^{\varepsilon}_{\lambda^\varepsilon_s}) \ ds$ and $\int_{t_{k-1}}^{t} b^\prime(r(s)) \left(R^\eps_{\lambda^\eps_s} - r(s)\right) \thinspace ds$, and using Taylor's formula, we get			
	\begin{multline}\label{E:R-error-k-clt}
		R^{\varepsilon}_{\lambda^\varepsilon_t} - r(t) - \eps R^1_t = \left\{R^{\varepsilon,+}_{\tau^\varepsilon_{k-1}} - r^+(t_{k-1}) -\eps R^{1,+}_{t_{k-1}}\right\} + \hat{\II}^\eps_k(t), \qquad \text{with} \\
		\hat{\II}^\eps_k(t) \triangleq \int_{t_{k-1}}^{t} b(R^{\varepsilon}_{\lambda^\varepsilon_s}) (J^\eps_{s}-1) \ ds + \int_{t_{k-1}}^{t} \frac{b^{\prime\prime}(\hat\xi^{k,\varepsilon}_s)}{2}\left(R^{\varepsilon}_{\lambda^\varepsilon_s}-r(s)\right)^2 \ ds \\+ \int_{t_{k-1}}^{t} b^{\prime}(r(s))\left(R^\eps_{\lambda^\eps_s} - r(s)-\eps R^1_s\right)\ ds +\varepsilon (W_{\lambda^\varepsilon_t}-W_{\tau^\varepsilon_{k-1}})-\eps (W_t - W_{t_{k-1}})
	\end{multline}
	where $\hat\xi^{k,\varepsilon}_s$ is a point between $r(s)$ and $R^\varepsilon_{\lambda^\varepsilon_s}$. As $t \nearrow t_k$, $1 \le k \le \sfN$, the expression above approaches $R^{\eps,-}_{\tau^\eps_k}-r^-(t_k) - \eps R^{1,-}_{t_k}$. The resetting rules in \eqref{E:det-state-int-eq}, \eqref{E:stoch-state-int-eq} and \eqref{E:fluct-proc} imply that $R^{\varepsilon,+}_{\tau^\varepsilon_{k}} - r^+(t_{k})-\eps R^{1,+}_{t_k} =h(R^{\varepsilon,-}_{\tau^\varepsilon_{k}})-h(r^-(t_{k}))- \eps h^{\prime}(r^{-}(t_{k})) R^{1,-}_{t_k}$. If we now add and subtract $h^\prime(r^-(t_k))(R^{\eps,-}_{\tau^\eps_k}-r^-(t_k))$, and then use Taylor's formula, we get
	\begin{equation}\label{E:R-resetting-k-clt}
			R^{\varepsilon,+}_{\tau^\varepsilon_{k}} - r^+(t_{k})-\eps R^{1,+}_{t_k}
			=\frac{h^{\prime\prime}(\hat\eta^{k,\varepsilon})}{2}(R^{\varepsilon,-}_{\tau^\varepsilon_{k}}-r^-(t_{k}))^2 + h^\prime(r^-(t_{k}))(R^{\eps, -}_{\tau^\eps_k} - r^-(t_k)-\eps R^{1,-}_{t_k}), 
	\end{equation} 
	where $\hat\eta^{k,\varepsilon}$ is a point between $r^-(t_{k})$ and $R^{\varepsilon,-}_{\tau^\varepsilon_{k}}$. We now proceed in a manner similar to that in Lemma \ref{L:LLN-comparison-G}. Starting from $[t_0,t_1)$, we work our way forward in time alternately applying \eqref{E:R-error-k-clt} and  \eqref{E:R-resetting-k-clt} in succession. As a result, we obtain that for $t \in [t_{k-1},t_k)$ with $1 \le k \le \sfN$, or $t \in [t_{k-1},\sfT]$ with $k=\sfN$+1, we have
	\begin{multline}\label{E:CLT-decomp-basic}
		R^\eps_{\lambda^\eps_t}-r(t) - \eps R^1_t = \prod_{j=1}^{k-1}h^\prime(r^-(t_j))\left(R^{\eps,+}_{\tau^\eps_0}-r^+(t_0) - \eps R^{1,+}_{t_0}\right) \\+ \sum_{i=1}^{k-1} \left(\prod_{j=i+1}^{k-1} h^\prime(r^-(t_j))\right) \frac{h^{\prime\prime}(\hat\eta^{i,\eps})}{2}\left(R^{\eps,-}_{\tau^\eps_i}-r^-(t_i)\right)^2 
		+ \sum_{i=1}^{k-1}\left(\prod_{j=i}^{k-1} h^\prime(r^-(t_j))\right) \hat{\II}^\eps_i(t_i) + \hat{\II}^\eps_k(t)
	\end{multline}
	where the quantities $\hat{\II}^\eps_i(t)$ are as in equation \eqref{E:R-error-k-clt}, and we use the conventions $\sum_{i=1}^0 (\dots) \triangleq 0$ and $\prod_{j=1}^0 (\dots) \triangleq 1$.  Carefully expanding the right-hand side of \eqref{E:CLT-decomp-basic} using the expressions for $\hat\II^\eps_i(t)$ from \eqref{E:R-error-k-clt}, we get
	\begin{multline}\label{E:CLT-decomp-detailed}
		R^\eps_{\lambda^\eps_t}-r(t) - \eps R^1_t =\prod_{j=1}^{k-1}h^\prime(r^-(t_j))\left(R^{\eps,+}_{\tau^\eps_0}-r^+(t_0) - \eps R^{1,+}_{t_0}\right) \\
		+ \sum_{i=1}^{k-1} \left(\prod_{j=i+1}^{k-1} h^\prime(r^-(t_j))\right) \frac{h^{\prime\prime}(\hat\eta^{i,\eps})}{2}\left(R^{\eps,-}_{\tau^\eps_i}-r^-(t_i)\right)^2 + \hat{\JJ}^\eps_k(t) + \hat{\LL}^\eps_k(t) + \hat{\MM}^\eps_k(t) + \hat{\NN}^\eps_k(t) \quad \text{where}
	\end{multline} 
	\begin{equation}\label{E:JLMN}
		\begin{aligned}
			\hat\JJ^\eps_k(t) &\triangleq \sum_{i=1}^{k-1} \left(\prod_{j=i}^{k-1} h^\prime(r^-(t_{j}))\right) \int_{t_{i-1}}^{t_i} b(R^\eps_{\lambda^\eps_s}) (J^\eps_{s}-1) \thinspace ds + \int_{t_{k-1}}^t b(R^\eps_{\lambda^\eps_s})\left(J^\eps_{s} - 1 \right) \thinspace ds \\ 
			\hat\LL^\eps_k(t) &\triangleq  \sum_{i=1}^{k-1} \left(\prod_{j=i}^{k-1} h^\prime(r^-(t_{j}))\right) \int_{t_{i-1}}^{t_i} \frac{b^{\prime\prime}(\hat\xi^{i,\eps}_s)}{2} \left(R^\eps_{\lambda^\eps_s}-r(s)\right)^2 \thinspace ds + \int_{t_{k-1}}^t \frac{b^{\prime\prime}(\hat\xi^{k,\eps}_s)}{2}\left(R^\eps_{\lambda^\eps_s}-r(s)\right)^2 \thinspace ds \\
			\hat\MM^\eps_k(t) &\triangleq  \sum_{i=1}^{k-1} \left(\prod_{j=i}^{k-1} h^\prime(r^-(t_{j}))\right) \int_{t_{i-1}}^{t_i} b^{\prime}(r(s))\left(R^\eps_{\lambda^\eps_s}-r(s)-\eps R^1_s\right) \thinspace ds \\
			&\quad \quad +\int_{t_{k-1}}^t b^{\prime}(r(s))\left(R^\eps_{\lambda^\eps_s}-r(s)-\eps R^1_s\right) \thinspace ds
			\\
			\hat\NN^\eps_k(t) &\triangleq  \eps \sum_{i=1}^{k-1} \left(\prod_{j=i}^{k-1} h^\prime(r^-(t_{j}))\right) \left(W_{\tau^\eps_i}-W_{\tau^\eps_{i-1}}- W_{t_{i}}+W_{t_{i-1}}\right) + \eps\left(W_{\lambda^\eps_t}-W_{\tau^\eps_{k-1}} - W_{t}+ W_{t_{k-1}}\right).
		\end{aligned}
	\end{equation}
	
	Recalling Assumptions \ref{A:vf-wedge} and \ref{A:resetting-map}, we now carefully estimate the terms on the right-hand side of the first line in \eqref{E:R-error-G}. The first two terms can be bounded above by $(\frac{\|h^{\prime\prime}\|}{2} \vee \|h^\prime\|^{k-1}) k \sup_{[0, \sfT]} |R^\eps_{\lambda^\eps_t} - r(t)|^2$. The third term is zero on account of the initial conditions. 
	Next, we note that $|\JJ^\eps_k(t)| \le (\|h^\prime\|^{k-1}\vee 1)\sfK_b \int_0^t |J^\eps_{s}-1| \thinspace ds \le \sfK_b (\|h^\prime\|^{\sfN}\vee 1) \int_0^\sfT |J^\eps_{s}-1| \thinspace ds$.
	It is easy to see that $|\LL^\eps_k(t)| \le \sfK_b (\|h^\prime\|^{\sfN}\vee 1) \int_0^t \sup_{0 \le u \le s}|R^\eps_{\lambda^\eps_u}-r(u)|^2 \thinspace ds$ and $|\MM^\eps_k(t)|\le  \sfK_b (\|h^\prime\|^{\sfN}\vee 1) \int_{0}^{t} \sup_{0\le u\le s} |R^\eps_{\lambda^\eps_u}- r(u)- \eps R^1_u|\ ds$. Finally, $|\NN^\eps_k(t)| \le 2k \eps (\|h^\prime\|^{\sfN}\vee 1) \sup_{\substack{0\le s<t\le\sfT \\ |t-s|\le \delta}} |W_t-W_s|$. Thus, there exists a constant $C>0$ (depending on $\sfK_b$, $\|h^\prime\|$, $\|h^{\prime\prime}\|$, $\sfT$, $\sfN$) such that
	
	\begin{equation*}
		\begin{aligned}
			|R^\eps_{\lambda^\eps_t}-r(t)- \eps R^1_t| \le  & C\left\{\|J^\eps_\cdot-1\|_\infty+ \sup_{[0, \sfT]} |R^\eps_{\lambda^\eps_t} - r(t)|^2 + \int_0^t \sup_{0 \le u \le s}|R^\eps_{\lambda^\eps_u}-r(u)|^2 \thinspace ds \right. \\ & \left.  + \int_0^t \sup_{0 \le u \le s}|R^\eps_{\lambda^\eps_u}-r(u)-\eps R^1_u| \thinspace ds+ \eps \sup_{\substack{0\le s<t\le\sfT \\ |t-s|\le \delta}} |W_t-W_s|\right\} .
		\end{aligned}
	\end{equation*}
	Since the right-hand side of the latter is non-decreasing in $t$, we have
	\begin{equation*}
		\begin{aligned}
			\sup_{0 \le s \le t} |R^\eps_{\lambda^\eps_s}-r(s)-\eps R^1_s|  \le & C\left\{\|J^\eps_\cdot-1\|_\infty + \sup_{[0, \sfT]} |R^\eps_{\lambda^\eps_t} - r(t)|^2 + \int_0^t \sup_{0 \le u \le s}|R^\eps_{\lambda^\eps_u}-r(u)|^2 \thinspace ds \right. \\ & \left.  + \int_0^t \sup_{0 \le u \le s}|R^\eps_{\lambda^\eps_u}-r(u)-\eps R^1_u| \thinspace ds + \eps \sup_{\substack{0\le s<t\le\sfT \\ |t-s|\le \delta}} |W_t-W_s|\right\}.
		\end{aligned}
	\end{equation*}
	Applying Gronwall's inequality to the above equation, yields,
	\begin{multline*}
		\sup_{0 \le s \le t} |R^\eps_{\lambda^\eps_s}-r(s)-\eps R^1_s|\\ \le C\left\{\|J^\eps_\cdot-1\|_\infty+ \sup_{[0, \sfT]} |R^\eps_{\lambda^\eps_t} - r(t)|^2+ \int_0^t \sup_{0 \le u \le s}|R^\eps_{\lambda^\eps_u}-r(u)|^2 \thinspace ds +\eps \sup_{\substack{0\le s<t\le\sfT \\ |t-s|\le \delta}} |W_t-W_s|\right\}e^{Ct}
	\end{multline*}
	
\end{proof}

\begin{proof}[Proof of Lemma \ref{L:CLT-comparison-B}]

In attempt to prove the stated assertion for the radial coordinate in the Lemma \ref{L:CLT-comparison-B}, firstly, observe that as in the Lemma \ref{L:LLN-comparison-B}, 
$|R^\eps_{\lambda^\eps_t}-r(t)-\eps R^1_t|\ind_{\Omega\setminus G^\eps_\sfN} \le \left(|R^\eps_t|+|r(t)| + \eps |R^1_t|\right) \ind_{\Omega\setminus G^\eps_\sfN}$, where we of course already have noted  $\lambda^\eps_t(\omega) \equiv t$ for $\omega \notin G^\eps_\sfN$, we have $R^\eps_{\lambda^\eps_t}=R^\eps_t$ for $\omega \notin G^\eps_\sfN$.

Next we have done the hard work already estimating $|r(t)|$ and $|R^\eps_t|$ in Lemma \ref{L:LLN-comparison-B}. We only need to estimate $|R^1_t|$.

For the fluctuation process $R^1_t$, we may write, 
\begin{equation*}
	\begin{aligned}
		R^1_t&= \sum_{i=1}^{\ren(t)} \left(\prod_{j=i}^{\ren(t)} h^\prime(r^-(t_{j}))\right) \int_{t_{i-1}}^{t_i} b^\prime((r(s)) R^1_s\, ds + \int_{t_{\ren(t)}}^{t} b^\prime(r(s)) R^1_s\, ds \\
		&\quad +  \sum_{i=1}^{\ren(t)} \left(\prod_{j=i}^{\ren(t)} h^\prime(r^-(t_{j}))\right) (W_{t_i} - W_{t_{i-1}}) + (W_t -W_\ren(t))
	\end{aligned}
\end{equation*}

Furthermore, we observe the following,

\begin{multline*}
	|R^1_t| \le \left(\|h^\prime\|_\infty \vee 1\right)^{\ren(t)} \left[\int_{0}^{t} b^\prime(r(s)) R^1_s ds + 2(\ren(t)+1)\sup_{[0, \sfT]} |W_t| \right]\\
	\text{The Grownwall's inequality yields, }\\
	|R^1_t| \le 2(\ren(t)+1) \left(\|h^\prime\|_\infty \vee 1\right)^{\ren(t)}  e^{\sfK_b \left(\|h^\prime\|_\infty \vee 1\right)^{\ren(t)} t } \sup_{[0, \sfT]} |W_t|
\end{multline*}

 Which can be further written as $|R^1_t|\le  2(\ren(t)+1) e^{(2\lhp/\alpha + \sfK_b e^{\lhp t/\alpha})t}  \sup_{[0, \sfT]} |W_t| $

	\end{proof}

\begin{lemma}\label{L:existence}
The solutions $(r(t),\theta(t))$ and $(R^\eps_t,\Theta^\eps_t)$ to the integral equations \eqref{E:det-state-int-eq} and \eqref{E:stoch-state-int-eq} exist.
\end{lemma}

\begin{proof}
To obtain the solution to \eqref{E:det-state-int-eq}, we define a sequence of functions $\{(\sfr^n(t),\sfv^n(t))\}_{n=1}^\infty$ recursively by
$\sfr^n(t) = r_{n-1}^+ + \int_{t \wedge t_{n-1}}^t b(\sfr^n(s)) \thinspace ds$, $\sfv^n(t) = \theta_{n-1}^+ + [t-(t \wedge t_{n-1})]$, where $(r_0^+,\theta_0^+) = (r_0,\theta_0)$ and $(r_n^+,\theta_n^+) \triangleq \Delta(\sfr^{n}(t_n),\sfv^n(t_n))$ for $n \ge 1$, $t_n=n\alpha$ for $n \in \BZ^+$.
The state $(r(t),\theta(t))$ is obtained by putting together the pieces above according to
$r(t) \triangleq \sum_{n \ge 1} \sfr^n(t)\cdot \ind_{[t_{n-1},t_n)}(t)$, $\theta(t) \triangleq \sum_{n \ge 1} \sfv^n(t)\cdot \ind_{[t_{n-1},t_n)}(t)$.
The solution to \eqref{E:stoch-state-int-eq} is obtained in a similar manner: we define a sequence of stochastic processes $\{(\sfR^{n,\eps}_t,\sfV^{n,\eps}_t)\}_{n=1}^\infty$ via
$\sfR^{n,\eps}_t = r^{\eps,+}_{n-1} + \int_{t\wedge \tau^\eps_{n-1}}^t b(\sfR^{n,\eps}_s) \thinspace ds + \eps (W_t-W_{t\wedge \tau^\eps_{n-1}})$, $\sfV^{n,\eps}_t = \theta^{\eps,+}_{n-1} + \int_{t\wedge \tau^\eps_{n-1}}^t [1 + \eps f(\sfR^{n,\eps}_s,\sfV^{n,\eps}_s)] \thinspace ds + \eps^p \sigma (B_t-B_{t\wedge \tau^\eps_{n-1}})$ where $(r^{\eps,+}_0,\theta^{\eps,+}_0) \triangleq (r_0,\theta_0)$, $(r^{\eps,+}_n,\theta^{\eps,+}_n) \triangleq \Delta(\sfR^{n,\eps}_{\tau^\eps_n},\sfV^{n,\eps}_{\tau^\eps_n})$ for $n \ge 1$, $\tau^\eps_0 \triangleq 0$, $\tau^\eps_n \triangleq \inf\{t> \tau^\eps_{n-1}:\sfV^{n,\eps}_t=\alpha\}$ for $n \ge 1$.
The stochastic process $(R^\eps_t,\Theta^\eps_t)$ is now given by 
$R^\eps_t \triangleq \sum_{n \ge 1} \sfR^{n,\eps}_t \cdot \ind_{[\tau^\eps_{n-1},\tau^\eps_n)}(t)$, $\Theta^\eps_t \triangleq \sum_{n \ge 1} \sfV^{n,\eps}_t \cdot  \ind_{[\tau^\eps_{n-1},\tau^\eps_n)}(t)$.
\end{proof}


\bibliographystyle{alpha}
\bibliography{ImpactsLiterature}

	\end{document}